\documentclass[11pt]{amsart}

\usepackage{amsmath}

\usepackage{amssymb}
\usepackage{amsfonts}
\usepackage{amsthm}
\usepackage{enumerate}
\usepackage{bbm}
\usepackage[mathscr]{eucal}
\usepackage{graphicx}
\usepackage{verbatim}

\newtheorem{theorem}{Theorem}[section]
\newtheorem*{theorem*}{Theorem}
\newtheorem{lemma}[theorem]{Lemma}
\newtheorem{corollary}[theorem]{Corollary}
\newtheorem{proposition}[theorem]{Proposition}

\newtheorem{theoremx}{Theorem}

\theoremstyle{definition}

\theoremstyle{remark}

\usepackage{dsfont}

\def\nh{{\frac{n+1}{2}}}
\def\Ups{{\Upsilon}}
\def\ups{{\upsilon}}
\def\varups{{\varpi}}

\def\sF{{\mathscr {F}}}

\def\intslash{\rlap{\kern  .32em $\mspace {.5mu}\backslash$ }\int}
\def\qsl{{\rlap{\kern  .32em $\mspace {.5mu}\backslash$ }\int_{Q_x}}}

\def\emph#1{{\it #1 }}
\def\diam{{\text{\rm diam}}}

\def\cf{{\it cf}}

\def\gcd{{\text{\rm gcd}}}

\def\dist{{\text{\rm dist}}}

\def\supp{{\text{\rm supp }}}
\def\rad{{\text{\rm rad}}}

\def\inn#1#2{\langle#1,#2\rangle}

\def\noi{\noindent}

\def\card{\text{\rm card}}
\def\lc{\lesssim}
\def\gc{\gtrsim}

\def\ga{\gamma}

\def\Ga{\Gamma}

\def\eps{\varepsilon}
\def\ep{\epsilon}

\def\ka{\kappa}

\def\la{\lambda}
\def\La{\Lambda}

\def\fI{{\mathfrak {I}}}

\def\fK{{\mathfrak {K}}}

\def\fQ{{\mathfrak {Q}}}

\def\fV{{\mathfrak {V}}}
\def\fW{{\mathfrak {W}}}

\def\fn{{\mathfrak {n}}}

\def\bbE{{\mathbb {E}}}

\def\bbN{{\mathbb {N}}}

\def\bbP{{\mathbb {P}}}

\def\bbR{{\mathbb {R}}}

\def\bbT{{\mathbb {T}}}

\def\bbZ{{\mathbb {Z}}}
\def\bbone{{\mathbbm {1}}}

\def\cA{{\mathcal {A}}}
\def\cB{{\mathcal {B}}}
\def\cC{{\mathcal {C}}}

\def\cE{{\mathcal {E}}}
\def\cF{{\mathcal {F}}}

\def\cR{{\mathcal {R}}}
\def\cS{{\mathcal {S}}}

\def\cV{{\mathcal {V}}}

\font \roman = cmr10 at 10 true pt

\def\be#1{\begin{equation}\label{#1}}
\def\ee{\end{equation}}
\def\bas{\begin{align*}}
\def\eas{\end{align*}}
\def\bi{\begin{itemize}}
\def\ei{\end{itemize}}
\def\dist{{\hbox{\roman dist}}}
\def\supp{{\hbox{\roman supp}}}
\def\dim{{\hbox{\roman dim}}}

\def\eps{\varepsilon}
\def\emph#1{{\it #1}}
\def\textbf#1{{\bf #1}}
\def\intslash{\rlap{\kern  .32em $\mspace {.5mu}\backslash$ }\int}
\def\qsl{{\rlap{\kern  .32em $\mspace {.5mu}\backslash$ }\int_{Q_x}}}
\def\diam{{\text{\rm diam}}}

\def\Xint#1{\mathchoice
{\XXint\displaystyle\textstyle{#1}}%
{\XXint\textstyle\scriptstyle{#1}}%
{\XXint\scriptstyle\scriptscriptstyle{#1}}%
{\XXint\scriptscriptstyle\scriptscriptstyle{#1}}%
\!\int}
\def\XXint#1#2#3{{\setbox0=\hbox{$#1{#2#3}{\int}$ }
\vcenter{\hbox{$#2#3$ }}\kern-.6\wd0}}

\def\dashint{\Xint-}

\newcommand{\Be}{\begin{equation}}
\newcommand{\Ee}{\end{equation}}

\begin{document}
\title[Convolution powers of Salem measures]{Convolution powers of Salem
 measures \\ with applications}
\author{Xianghong Chen and Andreas Seeger}
\address{X. Chen\\Department of Mathematics\\University of Wisconsin-Madison\\Madison, WI 53706, USA}
\curraddr{Department of Mathematics\\University of Wisconsin-Milwaukee, Milwaukee, WI 53201-0413, USA}
\email{chen242@uwm.edu}
\address{A. Seeger\\Department of Mathematics\\University of Wisconsin-Madison\\Madison, WI 53706, USA}
\curraddr{}
\email{seeger@math.wisc.edu}
\thanks{Supported in part  by the National Science Foundation}
\date{\today}

\begin{abstract}
We study the regularity of convolution powers for measures supported on
Salem sets, and  prove related  results on Fourier restriction and Fourier multipliers.  In particular we show
that  for $\alpha$ of the form
${d}/{n}$, $n=2,3,\dots$ there exist $\alpha$-Salem measures for which the $L^2$ Fourier restriction theorem holds in the range $p\le \frac{2d}{2d-\alpha}$.
 The results rely on ideas of  K\"orner.
We extend some of his constructions to obtain upper regular $\alpha$-Salem measures, with sharp regularity results for $n$-fold convolutions for all $n\in \bbN$.
\end{abstract}
\maketitle

\section{Introduction}
Given a finite positive  Borel  measure $\mu$  on $\bbR^d$ satisfying
the condition $$|\widehat \mu(\xi)|=O(|\xi|^{-b})$$
for some $b>0$, the Fourier transform maps
$L^p(\bbR^d)$ to $L^2(d\mu)$
for some $p>1$. This is the Fourier restriction  phenomenon discovered by Stein
in the 1960's. Much research in Fourier analysis has been done
regarding the case of $\mu$ being surface measure on the sphere where sharp
results are due to Tomas and Stein \cite{tomas}, \cite{stein}.
A general version of
Tomas' theorem is due to Mockenhaupt \cite{mockenhaupt} and also
Mitsis \cite{mitsis}. These authors showed that  under the above assumption
and the
additional regularity condition   $$\mu(B)=O(\diam(B)^a)$$
for all balls $B$ the Fourier transform maps
$L^p(\bbR^d)$ to $L^2(d\mu)$  for $1\le p<p_{a,b}= \frac{2(d-a+b)}{2(d-a)+b}$. It was shown in \cite{bak-seeger}  that
 the result is also  valid for $p=p_{a,b}$.
The  Fourier decay assumption implies that the
regularity condition holds for $a=b$.
Moreover, If the support of $\mu$ is contained in a set of Hausdorff dimension $\alpha$ then $b\le \alpha/2$, and $a\le \alpha$.
See
\cite[ch.8]{wolff}, \cite{mitsis} for these facts.
Of particular interest are measures supported on sets $E$
of Hausdorff dimension $\alpha$
for which the Fourier decay condition holds for all $b<\alpha/2$;
such sets are commonly called {\it Salem sets}.
 The existence of Salem sets is  due to
Salem \cite{salem}, see also the book by Kahane  \cite{kahane}, and  papers by Kaufman \cite{kaufman}, Bluhm (\cite{bluhm1}, \cite{bluhm2}) and
 \L aba and Pramanik \cite{laba-pramanik} for other constructions.

Here we are also interested in the special  Salem sets $E$ which carry probability measures
for which  the endpoint bound $|\widehat \mu(\xi)|= O(|\xi|^{-\dim(E)/2})$ holds  for large $\xi$, and make the following definition.

\smallskip

\noindent {\bf Definition.}
{\it
(i) A
Borel probability  measure $\mu$  is called an {\it $\alpha$-Salem measure} if
it is compactly supported, the support of $\mu$
is contained in a set of Hausdorff dimension $\alpha$, and if
\Be
\label{intro:sharpFourierdecay}
\sup_{\xi\in \bbR^d} (1+|\xi|)^{\alpha/2}|\widehat\mu(\xi)|<\infty\,.
\Ee

(ii) An $\alpha$-Salem measure is called  upper regular (or $\alpha$-upper regular)
if
\Be\label{intro:alphareg}
\sup_{B} \frac{\mu(B)}{\diam(B)^\alpha} <\infty
\Ee
where the sup is taken over all balls.

}

Examples of upper regular $\alpha$-Salem measures  were  constructed by
K\"orner (\cf. \cite{koerner1978}),
see also the work by the first author \cite{chenn} for various  refinements.

If $\mu$ is an upper regular $\alpha$-Salem measure then
the Fourier transform maps
$L^p(\bbR^d)$ to $L^2(\mu)$
for $1\le p \le \frac{4d-2\alpha}{4d-3\alpha}$,
by the result in \cite{bak-seeger}.
In analogy to results and conjectures for surface measure on the sphere,  Mockenhaupt
conjectured that the Fourier transform should map
$L^p(\bbR^d)$ to $L^1(\mu)$  for the larger range
$1\le p< \frac{2d}{2d-\alpha}$.
By 
\cite[Prop. 3.1]{mitsis} such an $L^p\to L^2$ result  cannot hold for $p>\frac{2d}{2d-\alpha}$.
Furthermore, he remarked that for suitable examples there is a  possibility that even  the
 stronger Stein-Tomas $L^p\to L^2(\mu)$ bound could hold in this range.
Recently Hambrook and \L aba \cite{hambrook-laba}
 gave,  for a dense set of $\alpha$'s (and $d=1$),
 examples of  Salem sets of dimension $\alpha$,  which show that
 the $p$ range
 for the $L^p\to L^2(\mu)$ bound in \cite{bak-seeger} cannot be improved in general.
Their examples carry randomness and arithmetic
structures at different scales. The first author  \cite{chenn} has extended this idea to provide, among other things,
for all $\alpha \in [0,1]$
examples of  upper regular $\alpha$-Salem measures on the real line,
 for which $\cF$ does not map $L^p$ to $L^2(\mu)$ for any $p> \frac{4d-2\alpha}{4d-3\alpha}$.

 These examples still do not exclude the Mockenhaupt scenario of a
 larger $p$-range
 for the $L^2$ restriction estimate for other types of Salem measures.
The question was explicitly posed  in a recent survey paper by
\L aba \cite{laba-icm}.
We show an optimal result when
 $\alpha$ is of the form $d/n$ with some integer $n$.

\begin{theoremx}\label{intro:theorem-A}
 Given $\alpha={d}/{n}$ where $n\in\mathbb N$, $n\ge 2$, there exists
an upper regular $\alpha$-Salem measure
so that $\cF: L^p(\bbR^d)\to L^2(\mu)$ is bounded in the optimal range   $1\le p\le  \frac{2d}{2d-\alpha}$.
\end{theoremx}

\noi{\it Remarks.}

(i) Shmerkin and Suomala \cite{shmerkin-suomala} have, independently, obtained a similar result, for $d=1$, $\alpha>1/2$. Their method  also covers the cases $d=2,3$, $d/2<\alpha\le 2$.  Their approach is quite different from the methods used here.

(ii)
 It would be of great interest to find Ahlfors-David regular $\alpha$-Salem measures, i.e. besides
\eqref{intro:sharpFourierdecay}, \eqref{intro:alphareg}  we would also have
a lower bound $\mu(B)\gc \rad(B)^{\alpha}$ for all balls  $B$ with radius $\le 1$ which are  centered in the support of $\mu$. This question has been raised  by Mitsis \cite{mitsis}, see also the list of problems in Mattila \cite{mattilaproblems}.
We remark that the examples by
Shmerkin and Suomala \cite{shmerkin-suomala} for the non-endpoint $L^2\to L^4$ restriction estimate  (with $\alpha>1/2$)  are  Ahlfors-David regular.
However the measures satisfying Theorem \ref{intro:theorem-A}
are
 necessarily {\it not} Ahlfors-David $\alpha$-regular, see
\S\ref{sec:restrmult}.


\medskip

 A variant of Theorem \ref{intro:theorem-A}
 can be used to derive some new results on a
class of Fourier  multipliers of Bochner-Riesz type as considered by
Mockenhaupt \cite{mockenhaupt}.
In what follows we let $M_p^q$ to be the space of all $m\in \cS'(\bbR^d)$ for which
$f\mapsto \cF^{-1}[m\widehat f]$ extends to a bounded operator from $L^p(\bbR^d)$ to $L^q(\bbR^d)$. The norm on $M^q_p$ is the operator norm, i.e.
$$\|m\|_{M^q_p}=\sup_{\substack {f\in \cS(\bbR^d)\\ \|f\|_p \le 1 }}
\big\|\cF^{-1}[m \widehat f]\big\|_q  \,.$$
In \cite{mockenhaupt} Mockenhaupt introduced a class of Fourier multipliers
 associated with general measures  which reflect the properties of
Bochner-Riesz multipliers in the case when  $\mu$ is the surface measure on a smooth hypersurface.

Given a compactly supported  $\alpha$-upper regular
 Borel measure, $\la>\alpha-d$ and $\chi\in C^\infty_c (\bbR^d)$ the function
\Be\label{intro:mla}
m_\la(\xi)=\int_{\mathbb R^d}\chi(\xi-\eta)|\xi-\eta|^{\la-\alpha}d\mu(\eta)
\Ee
is well defined as an $L^1$ function.
In  \S\ref{sec:restrmult} we prove among other things

\begin{theoremx}\label{intro:multiplier}
Let  $\alpha={d}/{n}$ where $n\in\mathbb N$, $n\ge 2$ and $\lambda>0$. There exists
an upper regular $\alpha$-Salem measure on $\bbR^d$
 so that for $1\le p<\frac{2d}{2d-\alpha}$, $p\le q\le 2$
we have $$m_\la\in M_p^q \,\,
\iff \,\,\la \ge
 d\big(\frac 1q-\frac 12\big)- \frac{d-\alpha}2.$$
\end{theoremx}

\medskip

Let $\mu^{*n}$ be the 
convolution of $n$ copies of $\mu$; more precisely we set  $\mu^{*0}=\delta_0$ (the Dirac measure at $0$), $\mu^{*1}=\mu$ and
$$\mu^{*n}=\mu*\mu^{*(n-1)}$$ for $n\ge 2$.
The proof of  the Fourier restriction result of Theorem  \ref{intro:theorem-A}
for $\alpha=d/n$
is based  on a regularity result for
the self convolutions  of suitable Salem measures and the
inequality
\Be\label{intro:conv-n}
\int|\widehat {g\mu}|^{2n} d\xi \lc
\|\mu^{*n}\|_\infty  \Big(\int|g(x)|^2 d\mu\Big)^n\,.
\Ee
\eqref{intro:conv-n} is  a special case  of an inequality in  \cite{chen},
closely related to a result
by Rudin \cite{rudin}.

For $n=2$,
 K\"{o}rner \cite{koerner2008}  proved
 the existence of a compactly supported probability measure on  $\bbR$,
supported on
a set of Hausdorff dimension $1/2$  for which $\mu*\mu$ is a continuous
function.
Moreover,
given $\frac{1}{2}\le\alpha<1$, there exists a Borel probability measure $\mu$ on $\mathbb R$
supported on a compact set of Hausdorff dimension $\alpha$
such that $\mu*\mu\in C_c^{\alpha-1/2}(\mathbb R)$.
These substantially improved  and extended
previous results by Wiener-Wintner \cite{ww} and Saeki \cite{saeki} on convolution squares for singular measures. Note that  by taking adjoints  inequality \eqref{intro:conv-n} for $n=2$ shows that $\cF: L^{4/3}\to L^2(\mu)$;
for $\alpha<2/3$ this  yields a range  larger
than  $[1, \frac{4-2\alpha}{4-3\alpha}]$, the largest range that could be
proved from \cite{bak-seeger}.
It is not stated in K\"orner's paper that the measures constructed there
have the appropriate  Fourier  decay  properties
but as  we shall  see this is not hard  to accomplish.

For integers $n\ge 0$ let $C^n(\bbR^d)$ be the space of functions whose derivatives up to order $n$ are continuous and bounded; the norm is given by
$$\|f\|_{C^n}=\sum_{|\alpha|\le n}\|\partial^\alpha f\|_\infty.$$
Let
$\psi:[0,\infty)\rightarrow [0,\infty)$ be  a nondecreasing  bounded
function  satisfying
\Be\label{psilimit}
\lim_{t\to 0}t^{-\eps}\psi(t)=\infty, \quad \forall \eps>0
\Ee
and, for some $C_\psi>0$,
\Be\label{psidoubling}
\psi(t)\le C_\psi\psi(t/2), \quad t>0.
\Ee
For a function $f$ on $\mathbb R^d$, define
\Be\label{eq:omega}\omega_{\rho,\psi}(f)=\sup_{\substack{x,y\in\mathbb R^d \\ x\neq y}}
\frac{|f(x)-f(y)|}{|x-y|^\rho\psi(|x-y|)}\Ee
and
$$C^{\rho,\psi}(\mathbb R^d)=\{f\in C(\mathbb R^d): \omega_{\rho,\psi}(f)<\infty\}.$$
If $\rho\ge 1$, define
$$C^{\rho,\psi}(\mathbb R^d)=\{f\in C^{\lfloor \rho \rfloor}(\mathbb R^d): \partial^{\beta} f\in C^{\rho - {\lfloor \rho \rfloor},\psi}(\mathbb R^d), |\beta|={\lfloor \rho \rfloor} \}.$$

For $0<\rho<1$ the choice of $\psi(t)=1$ yields the usual H\"older spaces.
Only the definition of $\psi$ for small $t$ is relevant.
Other suitable choices for $\psi$ are
(i) $\psi(t)= \exp(-\sqrt{\log t^{-1}})$ for $t\le e^{-1}$,
(ii)  $\psi(t)= 1/(\log t^{-1})$ for $t\le e^{-1}$,  or
(iii)  $\psi(t)= 1/(\log \log t^{-1})$ for $t\le e^{-e}$.

We extend    K\"{o}rner's
constructions to prove the following result for higher convolution powers of upper regular $\alpha$-Salem measures.

\begin{theoremx}\label{intro:theorem-C}
Given $d\ge 1$ and $0<\alpha<d$, there exists a Borel probability measure $\mu$ on $\mathbb R^d$
satisfying the following properties.

(i) $\mu$ is supported on a compact set of Hausdorff and lower Minkowski dimension $\alpha$.

(ii)  For all $\xi\in\mathbb R^d, |\xi|\ge 1$,
$$|\widehat\mu(\xi)|\lesssim \psi(|\xi |^{-1})|\xi|^{-\alpha/2}.$$

(iii)  For all $x\in\mathbb R^d, 0<r<1$, $1\le n<d/\alpha
$,
$$\mu^{*n}(B(x,r))\lesssim \psi(r) r^{n\alpha}\,.$$

(iv) For $n\ge d/\alpha$,
$$\mu^{*n}\in C^{\frac{n\alpha-d}{2},\psi}_c(\mathbb R^d).$$
\end{theoremx}



Note that under the dimensional restriction the Fourier decay exponent,
the upper regularity exponents $n\alpha$  and the H\"older exponent
$\frac{n\alpha-d}2$ for $\mu^{*n}$ are all optimal (cf. \S \ref{optimalityHoelder} below for the latter).

\emph{Notation}.  We write $\square_1\lesssim \square_2$ to indicate that $\square_1\le C\square_2$ for some constant $0<C<\infty$ independent
of the testing inputs which will usually be clear from the context.
For a measurable subset $E$ of $\bbR^d$ or $\bbT^d$ we let $|E|$ denote  the Lebesgue measure of $E$.


\emph{Structure of the paper.}
The proof of Theorem \ref{intro:theorem-C}
is given in the next two sections.
The restriction and multiplier theorems are considered in \S
\ref{sec:restrmult}.

\section{K\"orner's Baire category approach}
This section contains the extensions of K\"orner's arguments adapted and extended to yield
Theorem \ref{intro:theorem-C}. The results will be stated in the periodic setting and followed by a relatively straightforward  transference argument.

To fix notations, we write $\mathbb T=\mathbb R/\mathbb Z$ and $\mathbb T^d=\mathbb T\times\cdots\times\mathbb T$. We occasionally denote by $\la$ the uniform probability measure on $\mathbb T^d$. $\la$ is usually identified with the function $1$ and we shall also identify a continuous function $g$ with the measure $g\la$.  A subset $J\subset\mathbb T$ is called an interval if it is connected. A rectangle is of the form
$R=J_1\times\cdots\times J_d$ where $J_i$ are intervals; $R$ is called a cube if these intervals have the same length.
If $\mu$ is a finite Borel measure on $\mathbb T^d$, the Fourier transform of $\mu$ is defined as
$$\widehat\mu(r)=\int_{\mathbb T^d}e^{-2\pi i r\cdot t}d\mu(t)$$
where $r\in\mathbb Z^d$. Here as usual we have identified $\mathbb T^d$ with $[0,1)^d$. Note that $\widehat\mu(0)=\mu(\mathbb T^d)$ and $\widehat{\mathbb\la}(r)=\delta_0(r)$. Let $\mu$ and $\nu$ be two finite Borel measures on $\mathbb T^d$, $\mu*\nu$ is the finite Borel measure on $\mathbb T^d$ with Fourier transform $\hat\mu(r)\hat\nu(r)$.
Finally, we equip $\mathbb T^d$ with the usual group structure and the intrinsic metric which will be denoted by $$|x-y|:=\Big( \sum_{i=1}^d |x_i-y_i|^2\Big)^{1/2}$$
where $x=(x_1,\cdots,x_d),y=(y_1,\cdots,y_d)$ and $|x_i-y_i|$ denotes the intrinsic metric on $\mathbb T$. We will also fix an orientation of $\mathbb T$ so that derivatives
are uniquely defined.
With this distance the expression $\omega_{\rho,\psi}(f)$  in \eqref{eq:omega}
and the spaces $C^{\rho,\psi}$
can be defined in the same way on $\bbT^d$.

For each integer $n\ge d/\alpha$ we  fix a finite smooth partition of unity on $\mathbb T^d$, indexed by $\imath\in \fI_n$
\Be\label{eq:partofunity}\mathcal O^{(n)}=\{\chi^{(n)}_\imath\}_{\imath\in \fI_n}
\Ee
so that each $\chi^{(n)}_\imath$ is supported on a cube of side length smaller than $(2n)^{-1}$.

\subsection{\it A metric space} Let $\fK$ be the collection of closed subsets of $\bbT^d$
which form a complete metric space with respect to the Hausdorff distance
\Be\label{Hausdorff}
\begin{aligned}d_{\fK}(K_1, K_2))&=
\sup_{x\in K_1}\dist (x, K_2)+\sup_{y\in K_2}\dist(y, K_1)\\&=
\sup_{x\in K_1}\inf_{y\in K_2}| x-y |+\sup_{y\in K_2}\inf_{x\in K_1}|x-y|;
\end{aligned}
\Ee
see e.g. \cite{stein-sh4}. We now consider metric spaces of pairs $(K,\mu)$ where $K$ is a compact subset of $\bbT^d$ and $\mu$ is a nonnegative Borel measure supported on $E$.
These measures are assumed to satisfy
\Be\label{FTd-decay}\lim_{|r|\to\infty} \frac{|r|^{\alpha/2}|\widehat\mu(r)|}{\psi(\tfrac 1{|r|})} =0,
\Ee
Moreover, for $n\ge d/\alpha$  and  for each $n$-tuple
$i=(\imath_1,\dots,\imath_n) \in \fI_n^n$, the $n$-fold convolution
$(\chi^{(n)}_{\imath_1}\mu)*\cdots*(\chi^{(n)}_{\imath_n}\mu)$ is absolutely continuous and we have
\begin{equation}\label{bounded-convolution}
\big(\chi^{(n)}_{\imath_1}\mu\big)*\cdots*\big(\chi^{(n)}_{\imath_n}\mu\big)=g_{\mu,i}^{(n)}\la, \, \text{ with }
g_{\mu,i}^{(n)}\in C^{\frac {n\alpha-d }2,\psi}\,.
\end{equation}
We let $\fW$ be the set of all $(K,\mu)$ where $K\subset \bbT^d$ is closed,
$\mu$ is a nonnegative Borel measure supported in $K$ satisfying
\eqref{FTd-decay} and \eqref{bounded-convolution}. A metric on $\fW$ is given by
\begin{align}
\label{eq:Wmetric}
&d_{\fW}\big((K_1,\mu_1),(K_2,\mu_2)\big)
\\
\notag
&=d_{\fK}(K_1,K_2)+|\hat\mu_1(0)-\hat\mu_2(0)|+\sup_{r\in\mathbb Z^d\backslash\{0\}}
\frac{|r|^{\alpha/2}|\hat\mu_1(r)-\hat\mu_2(r)|}{\psi(|r|^{-1})}
\\
\notag
&\quad+
\sum_{n\ge d/\alpha} 2^{-n}\min\Big\{1,\sum_{i\in \fI_n^n}\|g_{\mu_1,i}^{(n)}-g_{\mu_2,i}^{(n)}\|_{C^{\frac{n\alpha-d}{2},\psi}}\Big\}.
\end{align}

\begin{lemma} \label{le:fW-metric}
(i)  $(\fW, d_\fW)$ is a complete metric space.

(ii) For every nonnegative $C^\infty$ function $f$ and every compact set $K$ such that $K\supset \supp(f)$  the pair $(K,f)$ belongs to $\fW$.

(iii) Let $\fV$ be the subspace of $\fW$ consisting of $(K,\mu)$  satisfying
\begin{equation}\label{bounded-convolution-low}
\mu^{*n}(Q)\le \psi(|Q|)|Q|^{n\alpha/d}
\end{equation}
for all cubes $Q$ and $1\le n<d/\alpha$. Then $\fV$
(with the metric inherited from $\fW$) is a closed subspace
of $\fW$.

(iv) Let $\fV_0$ be
the subset of $\fV$ consisting of
pairs $(K, g)\in \fV$ with
$g\in C^\infty(\bbT^d)$ and let
$\overline \fV_0$ be the closure of $\fV_0$  in $\fV$ with respect to the metric $d_{\fW}$.
Then $\overline \fV_0$ is a complete metric space and for every nonnegative $g\in C^\infty(\bbT^d)$ there is a $C>0$ so that for all compact $K\supset \supp(g)$ the pair $(K,g/C)$ belongs to  $\overline \fV_0$.
\end{lemma}

\begin{proof}
To identify a limit measure  of a Cauchy sequence the theorem of Banach-Alaoglu is used. The proof is a straightforward modification of the arguments in
\cite{koerner2008}, \cite{koerner}, see also \cite{chan}, \cite{koernerlecture}
and \cite{stein-sh4}.
\end{proof}

In order to prove a version of Theorem \ref{intro:theorem-C} we wish to show that
there are pairs $(K,\mu)\in \overline \fV_0$ such that $\mu$ is supported
in a set of lower Minkowski dimension and Hausdorff dimension $\alpha$.
This will be deduced from a Baire category argument, as follows.

\begin{theorem}\label{dense}
Suppose $\alpha<\gamma<d$ and $\varepsilon>0$. Let $\fV^{\gamma,\varepsilon}$ be the subset of $\overline {\fV_{0}}$ consisting of pairs $(K,\mu)$ for which there are cubes $Q_1,\cdots, Q_M$ with
\begin{equation}\label{covering-property}
K\subset\bigcup_{j=1}^M Q_j\ \ \text{ and }\ \ 
|Q_1|=\cdots=|Q_M|
<\varepsilon
M^{-d/\gamma}.
\end{equation}
Then
 $\fV^{\ga,\eps}$ is open and dense in $\overline \fV_0$.
\end{theorem}

The Baire category theorem gives

\begin{corollary}
$\bigcap_{N=1}^\infty \fV^{\alpha+1/N,1/N}$ is a dense $G_\delta$ set in
$\overline \fV_0$.
\end{corollary}

Let $\underline{\dim}_{\rm M}(K)$, $\dim_{\rm H}(K)$
 denote the lower
Minkowski dimension and Hausdorff dimension, respectively.
Then
$\dim_{\rm H}(K)\le
\underline{\dim}_{\rm M}(K)$. If   $(K,\mu)\in\bigcap_{N=1}^\infty \mathcal \fV^{\alpha+1/N,1/N}$, then
$\underline{\dim}_{\rm M}(K)\le\alpha$  and hence also
${\dim}_{\rm H}(K)\le\alpha$.
On the other hand, $\eqref{FTd-decay}$ implies $\dim_{\rm H}(K)\ge\alpha$ (see e.g. \cite[Corollary~8.7]{wolff}). Thus we obtain

\begin{corollary}\label{quasi-all}
The set of $(K,\mu)\in \overline \fV_0$ satisfying
$$\underline{\dim}_{\rm M}(K)=\dim_{\rm H}(K)=\alpha$$
is of second category in $\overline{\fV_0}$.
\end{corollary}

Concerning the proof of Theorem
\ref{dense}, it is easy to see that
the sets $\fV_{\ga,\eps}$  are open subsets of $\overline{\fV_0}$.  The remainder of this section is devoted to proving that they are dense.

\subsection{\it Averages of point masses}
For large $N$ let  $\Gamma_N$ be
the finite subgroup
of $\bbT$  of order $N$, consisting
of $\{k/N: k=0,1,\dots, N-1\}$. Let $\Gamma_N^d$ the $d$-fold product, a subgroup of $\bbT^d$.

The following result yields measures on $\bbT^d$  which are sums of point masses supported on points in $\Gamma_N^d$ and  satisfy properties analogous to  \eqref{FTd-decay}, \eqref{bounded-convolution} and \eqref{covering-property}.

\begin{proposition}\label{prop:point-masses}
Given $0<\beta<d$ and an integer $\mathfrak{n}\ge 2$,
there exist $N_0(\beta,\mathfrak{n})\ge 1$, $C_1=C_1(d)$, $C_2=C_2(\beta,d)$, $C_3=C_3(\beta, d, \fn)$
 such that for all $N\ge N_0(\beta,\mathfrak{n})$ with $\gcd(\mathfrak{n}!,N)=1$, $P:=\lfloor N^\beta\rfloor$
 there is a choice of $x_1,\dots, x_P$ with  $x_j\in \Gamma_N^d$, such
that the following properties hold for  the measure
$$\mu=\frac{1}{P}\sum_{j=1}^P \delta_{x_j}.$$
\begin{subequations}
(i) For all ${r\in \Gamma_N^d\setminus \{0\}}$,
\begin{equation}\label{eq:fourier-decay-final}
|\widehat\mu(Nr)|\le C_1N^{-\beta/2} (\log N)^{1/2}
\end{equation}

(ii) For $1\le \ell\le d/\beta$
and for all cubes $Q$ with
$|Q|\le N^{-\ell\beta}$,
\begin{equation}\label{eq:bdd-multiplicity}
 \mu^{*\ell}(Q)\le C_2 N^{-\ell \beta} \log N\,.
\end{equation}

(iii) For  $d/\beta\le \ell\le \mathfrak{n}$,
 \begin{equation}\label{eq:bdd-convol}
 \max_{u\in \Gamma_N^d}
\big|\mu^{*\ell}(\{u\})-N^{-d}\big|
\le C_3\frac{N^{-d}(\log N)^{\frac{\ell+1}{2}}}{N^{(\ell\beta-d)/2}}
\end{equation}
\end{subequations}

\end{proposition}

While this result is not optimal (in particular with respect to the powers of the logarithm),
it is all we need for the proof of Theorem \ref{dense}.
See  \S \ref{sec:point-masses}.

\subsection{\it Transference}\label{subsec:transference-principle}
For $N\ge 1$, we will write
$$\sqcap_N=N^d\mathds{1}_{[-1/2,1/2)^d}(Nt)dt$$
and
$$\tau_N=\frac{1}{N^d}\sum_{j\in \Gamma_N^d}\delta_{j/N}.$$ Recall that  $\la$ is  the uniform probability measure (i.e. normalized  Lebesgue measure) on $\bbT^d$.

We start with some simple observations.
\begin{lemma}\label{lemma:obs}
The following
holds true for $N\ge 1$:

(i) $\sqcap_N^{*\ell}*\tau_N=\la$ for $\ell=1,2,\cdots.$

(ii) $\widehat{\tau_N}(r)=1$ for $r\in (N\mathbb Z )^d$, and $\widehat{\tau_N}(r)=0$ otherwise.

(iii) $\widehat{\sqcap_N}(r)=0$ for $r\in (N\mathbb Z )^d, r\neq 0$.

\end{lemma}
\begin{proof}  (i) follows by direct computation of the convolution
(it is also a
consequence of (ii) and (iii)).
For (ii) notice  that if
$r\notin (N\mathbb Z )^d$,
$$\widehat{\tau_N}(r)=\frac{1}{N^d}\sum_{j\in [N]^d}e^{-2\pi ir\cdot j/N}=\prod_{k=1}^d \frac{1}{N}\frac{e^{-2\pi ir_k}-1}{e^{-2\pi ir_k/N}-1}=0.$$
Otherwise $\widehat{\tau_N}(r)=1$.
For (iii) just notice that
$\widehat{\sqcap_N}(r)=\prod_{k=1}^d\frac{\sin(\pi r_k/N)}{\pi r_k/N}$.
\end{proof}

In what follows we let $\ups$ be a nonnegative smooth function supported in
$(-1/2,1/2)^d$ such that $\int \ups(t) dt=1$, and let
$\ups_N= N^d\ups(N\cdot).$
Thus $\ups_N$ generate a standard smooth approximation of the identity.
We now convolve the point masses obtained in Proposition
 \ref{prop:point-masses} with $\sqcap_N$ and  the mollifier $\ups_N$.

\begin{lemma}\label{lemma:convol-2}
Let $\mu$ be as in Proposition \ref{prop:point-masses} and let $f= \ups_N*\sqcap_N*\mu$. Then $f$
is a smooth function satisfying the following properties.

(i)  For $l=0,1,\dots$,
\begin{equation*}
\|\nabla^l f\|_\infty\le C(l) N^{d+l}.\
\end{equation*}
 There are cubes  $Q_j,j=1,\cdots,\lfloor N^{\beta}\rfloor
$  with side length $2/N$ such that
\begin{equation*}
\emph{supp}(f)\subset \bigcup_{j=1} ^{\lfloor N^{\beta}\rfloor} Q_j\,.
\end{equation*}
\begin{subequations}

(ii) For  $r\in\mathbb Z^d\backslash\{0\}$, $\La\ge 0$
\begin{equation}\label{eq:fourier-decay-f}
|\widehat f(r)|\le C(\log N)^{-1/2} N^{-\beta/2}\min\Big(\frac{C_0 (\La)N^\La}{|r|^\La},1\Big).\
\end{equation}

(iii) For all cubes $Q$
\begin{equation}\label{eq:regularity-f}
\int_Q f^{*n}(t)dt \le 2^d |Q|^{n\beta/d} \log N, \quad 1\le n\le d/\beta\,.
\end{equation}

(iv)
For $l=0,1,2,\dots$
\begin{equation}\label{eq:Lip-1-f}
\|\nabla^l (f^{*n}-1)\|_\infty\le C(l)C(\beta,\mathfrak{n})\frac{(\log N)^\nh}{N^{(n\beta-d)/2}}N^l, \quad d/\beta\le n\le \fn.
\end{equation}

\end{subequations}
\end{lemma}


\begin{proof} The assertion about the support follows immediately from the definition. Let $$g(t)=\sqcap_N*\mu(t)=\int_{\mathbb T^d}\sqcap_N(t-s)d\mu(s).$$ The mollifiers satisfy $\ups_N(r)\le N^d \max\{1, C(\Lambda) (N/|r|)^{\Lambda}\}$ for any $\Lambda\ge 0$. We thus observe that the estimates for $f$ are implied by
the following estimates for $g$.
\begin{subequations}
\begin{equation}\label{eq:fourier-decay-g}
\sup_{r\in \mathbb Z^d\backslash\{0\}}|\hat g(r)|\le C(\log N)^{1/2} N^{-\beta/2}\,,
\end{equation}
\begin{equation}\label{eq:regularity-g}
\int_Q g^{*n}(t)dt \le 2^d |Q|^{n\beta/d} \log N,\  n\le d/\beta,
\end{equation}
for all cubes $Q$,
\begin{equation}\label{eq:bdd-convol-g}
\sup_{t\in\mathbb T^d}|g^{*n}(t)-1|\le C(\beta,\mathfrak{n})\frac{(\log N)^{\nh}}{N^{(n\beta-d)/2}}, \
d/\beta\le n\le \fn,
\end{equation}
and
\begin{equation}\label{eq:bdd-multiplicity-g}
\sup_{t\in\mathbb T^d}|g(t)|\le N^d\,.
\end{equation}
\end{subequations}

To show \eqref{eq:fourier-decay-g}, notice that
$\widehat g(r)=\widehat{\sqcap_N}(r)\widehat{\mu}(r).$
If $r\in (N\mathbb Z)^d$, then  $\widehat g(r)=0$, by Lemma \ref{lemma:obs}, (ii).
Otherwise use the trivial bound $|\widehat{\sqcap_N}(r)|\le 1$ and \eqref{eq:fourier-decay-final}, together with the observation that $\widehat\mu$ is $N$-periodic.

To show \eqref{eq:regularity-g}, we consider separately the three cases
$|Q|\le N^{-d}$,
$N^{-d}\le |Q|\le N^{-n\beta}$, and $|Q|\ge N^{-n\beta}$.

{\it Case 1: $|Q|\le N^{-d}$.}
Notice that, as in the proof of \eqref{eq:bdd-multiplicity-g}, we have
$$\sqcap_N*\mu^{*n}(t)=N^d \mu^{*n}(\{u\})\le N^d M(\beta)\frac{\log N}{N^{n\beta}}.$$
Thus
\begin{align*}
&\sqcap_N*\mu^{*n}(Q)
\le |Q| N^d M(\beta)\frac{\log N}{N^{n\beta}}\\
&= |Q|^{n\beta/d} (|Q|N^d)^{1-n\beta/d} M(\beta){\log N}
\le M(\beta)|Q|^{n\beta/d}{\log N}
\end{align*}
by our assumption on $|Q|$. \\

{\it Case 2: $N^{-d}\le |Q|\le {N^{-n\beta}}.$}
In this case, by \eqref{eq:bdd-multiplicity}
\begin{align*}
\int_Q g^{*n}(t)dt&= \int_Q\sqcap_N^{n}*\mu^{*n}(t)dt
\le \max_{Q: |Q|=
N^{-n\beta}} \mu^{*n}(Q)\\
&\le M(\beta)N^{-n\beta}\log N\le M(\beta)|Q|^{n\beta/d}{\log N}.
\end{align*}

{\it Case 3:  $|Q|\ge{N^{-n\beta}}.$}
In this case we can split $Q$ into no more than $2^dN^{n\beta} |Q|$
cubes of size at most $N^{-n\beta}$. Applying  \eqref{eq:bdd-multiplicity} to each cube we may bound $\mu^{*n}(Q)$ by
$$
 (2^dN^{n\beta} |Q|) M(\beta)\frac{\log N}{N^{n\beta}}
=2^d M(\beta) |Q| \log N
\le 2^d M(\beta) |Q|^{n\beta} \log N.
$$
Since $g=\sqcap_N^{n}*\mu^{*n}$, $\eqref{eq:regularity-g}$ follows also in Case 3.

To show \eqref{eq:bdd-convol-g}, notice that by Lemma \ref{lemma:obs} (i)  and \eqref{eq:bdd-convol},
$$g^{*n}
=\sqcap_N^{*n}*\tau_N+\sqcap_N^{*n}*(\mu^{*n}-\tau_N)=\la+\sqcap_N^{*n}*(\mu^{*n}-\tau_N)
$$
and
$$|\mu^{*n}-\tau_N|\le C(\beta,\mathfrak{n})\frac{(\log N)^{\nh}}{N^{(n\beta-d)/2}}\tau_N.$$
Now $g^{*n}$ is continuous and we get
$$|g^{*n}-1|
\le  C(\beta,\mathfrak{n})\frac{(\log N)^{\nh}}{N^{(n\beta-d)/2}}
$$
and thus
\eqref{eq:bdd-convol-g}.

To show \eqref{eq:bdd-multiplicity-g}, notice that for any $t\in\mathbb T$,
$g(t)=N^d\mu(\{u\})$
where $u$ is the unique point in $\Gamma_N^d$ contained in the cube
$(t-1/(2N),t+1/(2N)]^d$. Now \eqref{eq:bdd-multiplicity-g} follows from \eqref{eq:bdd-multiplicity} with $n=1$ and $Q$ containing $u$.
\end{proof}

\noi{\bf Definition.} Let $f$ be a smooth function on $\bbT^d$ and let $p\in \bbN$.  We let   the {\it $p$-periodization} $\text{Per}_p f$  be the unique smooth function on $\bbT^d$
which is $1/p$-periodic in each of the $d$ variables and satisfies
$$\text {Per}_pf(t)= f(pt)\quad  \text{  for $0\le t_i<p^{-1}$, $i=1,\dots, d$.}$$

The following lemma is analogous to a crucial observation about periodized function in \cite{koerner2008}.

\begin{lemma}
\label{lemma:periodprop} Let $p\in \bbN$.

(i) Let $f\in C^\infty(\bbT^d)$. Then
$$
\widehat {\text{\rm Per}_pf}(kp )=\widehat f(k), \text{ $k\in \bbZ^d$,}
$$
and
$\widehat {\text{\rm Per}_pf}(r)=0$ if $r\in \bbZ^d$ is not of this form.

(ii) Let  $\cR=[a_{1}, a_{1}+p) \times \cdots\times [a_{d}, a_{d}+p)$, for some
$a\in \bbR^d$ and, for $\nu=1,\dots n$ let
$P_\nu$ be a trigonometric polynomial with frequencies in $\cR$, i.e.
$P_\nu$ is a linear combination of the functions $x\mapsto  \exp(2\pi i \inn{k}{x})$ with $k\in \cR\cap \bbZ^d$.
Let $f_1, \dots, f_n$ be  smooth functions on $\bbT^d$ and let
$G_\nu=\text{\rm Per}_p f_\nu$. Then
$$(G_1P_1)*\dots *(G_nP_n)= (G_1*\cdots*G_n) (P_1*\cdots*P_n)\,.$$
\end{lemma}
\begin{proof} This follows easily by Fourier expansion using the fact  that every
$k\in \bbZ^d$ can be written in a unique way as $k=pl+k'$ where $l\in \bbZ^d$ and $k'\in \cR$.
\end{proof}


\begin{lemma}\label{lemma:multiplier-F}
Let $\eta>0$,
$\beta>\alpha$ and let  $k$ be an integer with  $k>\frac{\alpha+1}{\beta-\alpha}$.
Then there exists $m_0=m_0(\alpha,\beta,\mathfrak{n},\eta,\psi,k)\ge N_0(\beta,\mathfrak{n})$ such that for all $m\ge m_0$ with $\gcd(\mathfrak{n}!,m)=1$ the following hold with $N=m^k$ and $f$ as in Lemma \ref{lemma:convol-2}.
\begin{subequations}

  (i) The $(2m+1)$-periodization of $f$,
$$F_m=\text{\rm Per}_{2m+1} f,$$
is smooth with $\int_{\mathbb T^d}F_m(t)dt=1$, and, for $l=0,1,\dots, L$
\begin{equation}\label{eq:bdd-derivative-F}
\|\nabla^l (F_m)\|_\infty\le C(L) m^{kd+(k+1)l}\,.
\end{equation}
Moreover, there are cubes $Q_j$
$j=1,\cdots,(2m+1)^d\lfloor m^{k\beta}\rfloor$,
of side length $m^{-k-1}$,
  such that
\begin{equation}\label{eq:supp-F}
\emph{supp}(F_m)\subset \bigcup_{j=1} ^{(2m+1)^d\lfloor m^{k\beta}\rfloor}
Q_j\,.
\end{equation}
\end{subequations}
\begin{subequations}
(ii) For $r\in \mathbb Z^d\backslash\{0\}$,
\begin{equation}\label{eq:fourier-decay-F}
\frac{|r|^{\alpha/2}|\widehat {F_m}(r)|}{\psi(1/|r|)}\le \eta.
\end{equation}

(iii) For all cubes $Q$ with side length at most $2/\sqrt{m}$.
\begin{equation}\label{eq:regularity-F-1}
\int_Q F_m^{*n}(t)dt \le \eta \psi(|Q|)|Q|^{n\alpha/d},\ 1\le n<d/\alpha
\end{equation}

(iv) For $n\ge d/\alpha$ let $\rho_n=
\frac{n\alpha-d}{2}.$ Then
\begin{equation}\label{eq:bdd-convol-F}
\|F_m^{*n}-1\|_{C^{\rho_n,\psi}}\le \eta,\  d/\alpha\le n\le \fn,
\end{equation}

(v) For all rectangles $R$ of side lengths at least $1/\sqrt{m}$.
\begin{equation}\label{eq:regularity-F-2}
\int_R F_m^{*n}(t)dt \le (1+\eta)|R|,\ n< d/\alpha
\end{equation}

\end{subequations}
\end{lemma}

\begin{proof}
Part (i) is
straightforward given Lemma \ref{lemma:convol-2}.
We thus just need to give the proof of (ii).

 We first recall from Lemma \ref{lemma:periodprop} that
$$\widehat {F_m}\big((2m+1)k\big)=\hat f(k)$$
for $k\in \mathbb Z^d$, and $\widehat {F_m}(r)=0$ for $r$ not of this form. Thus for $r\neq 0$, by \eqref{eq:fourier-decay-f}
\begin{align}\notag
|\widehat {F_m}(r)|
&\le \frac{\sqrt{8}\log^{1/2}(8m^{kd})}{m^{k\beta/2}}
\min\Big(\frac{C(\La)m^{k\La}(2m+1)^\La}{|r|^\La},1\Big)\\
\notag
&\le C_{\La,k}\frac{\log^{1/2} m}{m^{k\beta/2}}\min\Big(\frac{m^{(k+1)\La}}{|r|^\La},1\Big)\\
\notag
&= C_{\La,k}\frac{\log^{1/2}m}{m^{(k(\beta-\alpha)-\alpha)/4}}\frac{\psi(m^{-k-1})^{-1}}{m^{(k(\beta-\alpha)-\alpha)/4}}
\frac{\psi(m^{-k-1})}{m^{(k+1)\alpha/2}}\min\Big(\frac{m^{(k+1)\La}}{|r|^\La},1\Big)\\
\label{beforedoubling}
&\le \eta \frac{\psi(m^{-k-1})^{-1}}{m^{(k(\beta-\alpha)-\alpha)/4}}
\frac{\psi(m^{-k-1})}{m^{(k+1)\alpha/2}}\min\Big(\frac{m^{(k+1)\La}}{|r|^\La},1\Big)
\end{align}
provided that $\ge m\ge m_0$  and $m_0$ is chosen large enough.
We separately consider  the cases $0<|r|\le m^{k+1}$ and $|r|\ge m^{k+1}$.
In the first case we obtain \eqref{eq:fourier-decay-F} directly from
\eqref{beforedoubling}, provided that $m_0$ is large enough.  Now let
$2^l\le r/m^{k+1}<2^{l+1}$ with $l\ge 0$. Then by the monotonicity of $\psi$
and the doubling condition \eqref{psidoubling},
$$\psi(m^{-k-1})\le \psi(2^{l+1}|r|^{-1}) \le C_\psi^{l+1} \psi(|r|^{-1})$$
and we see in this case
\eqref{beforedoubling} is estimated by
$$
 \eta \frac{\psi(m^{-k-1})^{-1}}{m^{(k(\beta-\alpha)-\alpha)/4}}
 2^{(l+1)\alpha} C_\psi^{l+1} 2^{-l\La} \psi(|r|^{-1})|r|^{-\alpha/2}.
 $$
 Thus if above we choose $\La$ so large that $2^{\alpha+2-\La} C_\psi\le 1$
 we may sum in $l$. Then  by choosing $m_0$ large we obtain
 \eqref{eq:fourier-decay-F} for all $r\neq 0$.

 \medskip

\noi{\it Proof of (iv).} Notice that by \eqref{eq:Lip-1-f} and our assumption on $k$,
$$\|F_m^{*n}-1\|_{C^{\lfloor\rho_n\rfloor}}\le N^{-(\rho_n-\lfloor\rho_n\rfloor)-\epsilon}$$
for some $\epsilon>0$ and sufficiently large $m$. Setting
$$G=\nabla^{\lfloor\rho_n\rfloor} \big(F_m^{*n}-1)$$
it remains to show
$$\omega_{\rho_n-\lfloor\rho_n\rfloor,\psi}(G)\le \eta/2$$
for $m\ge m_0$ and large enough $m_0$.

Again by \eqref{eq:Lip-1-f}, we have
\begin{align*}
\|G\|_\infty+ N^{-1}
\|\nabla G\|_\infty\le N^{-(\rho_n-\lfloor\rho_n\rfloor)-\epsilon}
\end{align*}
for some $\epsilon>0$ and sufficiently large $m$. Now if $0<|h|\le 1/N$, then by the mean value theorem, for any $x\in\mathbb T^d$,
\begin{align*}
&\frac{|G(x+h)-G(x)|}{|h|^{\rho_n-\lfloor\rho_n\rfloor}\psi(|h|)}
= \frac{|G(x+h)-G(x)|}{|h|} \frac{|h|^{1-(\rho_n-\lfloor\rho_n\rfloor)}}{\psi(|h|)}\\
&\quad\le {N^{1-(\rho_n-\lfloor\rho_n\rfloor)-\epsilon}} C_{\psi,\epsilon}{|h|^{1-(\rho_n-\lfloor\rho_n\rfloor)-\epsilon/2}}\le C_{\psi,\epsilon} N^{-\epsilon/2}\le \eta/2
\end{align*}
provided that $m_0$ is chosen large enough. If $|h|\ge 1/N$, then
\[\frac{|G(x+h)-G(x)|}{|h|^{\rho_n-\lfloor\rho_n\rfloor}\psi(|h|)}
\le \frac{2\|G\|_\infty}{|h|^{\rho_n-\lfloor\rho_n\rfloor}\psi(|h|)}
\le 2N^{-(\rho_n-\lfloor\rho_n\rfloor)-\epsilon}\frac{N^{\rho_n-\lfloor\rho_n\rfloor}}{\psi(1/N)}
\le \eta/2
\]
provided that $m_0$ is chosen large enough. This proves \eqref{eq:bdd-convol-F}.

\medskip

\noi{\it Proofs of (iii) and (v).}
In what follows we say that a {\it fundamental cube } is a cube of the form
$\prod_{i=1}^d [\frac{\nu_i}{2m+1}, \frac{\nu_i+1}{2m+1})$ where
$\nu_i\in\{0,\dots 2m\}$ for each $i=1,\dots,d$.

We first consider the claim (v). Let $R$ be a rectangle with side lengths
$l_1\ge \cdots\ge l_d$, and assume that $l_d\ge m^{-1/2}$.
 Notice that $R$ is contained in a union of no more than
$$(2m+1)^d l_1\cdots l_d+C_d (2m+1)^{d-1}l_1\cdots l_{d-1}$$
many fundamental cubes of size $1/(2m+1)^d$. Since the integral of $F_m^{*n}$ over any fundamental cube is equal to $(2m+1)^{-d}$, we see that
\begin{align*}
\int_I F_m^{*n}(x)dx
&\le l_1\cdots l_d+C_d (2m+1)^{-1}l_1\cdots l_{d-1}\\
&= |R|+ \frac{C_d}{(2m+1)l_d}|R|\le |R|+ \frac{C_d}{2\sqrt{m}}|R|.
\end{align*}
Thus  \eqref{eq:regularity-F-2} is satisfied if $m_0$ is chosen large enough.

In order to show  (iii)
we separately consider the two cases where
the side length of $Q$ is larger or smaller than $(2m+1)^{-1}$.

\medskip{\it Case 1: $(2m+1)^{-1}\le |Q|^{1/d}\le {2}m^{-1/2}.$}
In this case the argument above shows
$$\int_QF_m^{*n}(x)dx\le (1+C_d)|Q|$$ and
\eqref{eq:regularity-F-1} will follow if
$$(1+C_d)|Q|\le \eta \psi(|Q|)|Q|^{n\alpha/d}.$$
But this is  indeed the case if $|Q|\le 2/\sqrt m\le 2/\sqrt {m_0}$  and $m_0$ is large enough.

{\it Case 2: $|Q|^{1/d} <(2m+1)^{-1}.$}
We first assume  that $Q$ is contained in a  $[0,(2m+1)^{-1})^d$.
 Then by \eqref{eq:regularity-f}
\begin{align*}
\int_Q F_m^{*n}(x)dx
&=\frac{1}{(2m+1)^d}\int_{(2m+1)Q} f^{*n}(t)dt\\
&\le\frac{1}{(2m+1)^d} 2^d \big((2m+1)^d|Q|)\big)^{n\beta/d} \log N\\
&=\frac{2^d k \log m}{(2m+1)^{d-n\beta}} |Q|^{n\beta/d}
\le |Q|^{n\beta/d}
\end{align*}
provided that $m_0$ is chosen large enough. \eqref{eq:regularity-F-1} will follow if
$$|Q|^{n\beta/d}\le \eta\psi(|Q|)|Q|^{n\alpha/d},$$
But this is  the case if $|Q|^{1/d}\le 1/(2m+1)\le 1/m_0$ is small enough.
By periodicity the above argument holds true if $Q$ is contained in any fundamental cube of size $(2m+1)^{-d}$.   Moreover if $Q$ is any cube of size
$\le (2m+1)^{-d}$ then we may split $Q$ in $2^d$ rectangles supported in fundamental cubes and apply the same  argument to each such rectangle.
This finishes the proof
of \eqref{eq:regularity-F-1}.
\end{proof}

\subsection{\it Approximation}
We are now ready to prove Theorem \ref{dense}. It remains to show that,
for every $\ga\in (\alpha,d)$ and every $\eps_1>0$ the set
$\fV^{\gamma,\eps_1}$ is dense in $\overline {\fV_0}$.
This reduces to approximating $(K, g)\in \fV_0$ where $g$ is smooth.
We may further assume that there exists a small constant $c>0$ such that
$g$ satisfies
\begin{equation}\label{bounded-convolution-low-c}
\int_Q g^{*n}(x)dx\le (1-c)\psi(|Q|)|Q|^{n\alpha/d}
\end{equation}
for all cubes $Q$ and $1\le n <d/\alpha$. This is because otherwise we can approximate $(K,g)$ by $(K,(1-c)g)$ and let $c\rightarrow 0$.

\begin{lemma}\label{lemma:approx}
Suppose  $\alpha<\gamma<d$, $\varepsilon_1>0$, $c\in (0,1)$, $(K,g)\in \fV_0$ where $g$ is a smooth function satisfying \eqref{bounded-convolution-low-c}. Let $\epsilon>0$. Then there exists a compact set $F$ and a smooth function $f$ such that $(F,fg)\in \mathcal \fV^{\gamma,\varepsilon_1}$ and
$$d_{\fW}\big((K,g),(F,fg)\big)<\epsilon.$$
\end{lemma}

\begin{proof} We let $\ep'=\ep/100$.
Fix $\beta$ with $\alpha<\beta<\gamma$. Choose
$\fn\equiv\fn(\ep)=1+\lceil\log_2\tfrac{1}{\ep'}\rceil$ so that
\Be\label{choiceofn}
\sum_{n>\fn}2^{-n}<\ep'.
\Ee
Fix an integer $k$ such that $$k>\frac{d-\gamma}{\gamma-\beta}.$$

With these parameters we consider the functions $F_m$ as constructed in Lemma \ref{lemma:multiplier-F}. We let
$A_{\ep'}$ to be a finite  $\ep'$-net of $K$; i.e. a finite set of points in $K$ such that $K$ is contained in the union of balls of radius $\ep'$ centered at points in $A_{\ep'}$.
We shall show that  if $\eta>0$ is chosen small enough and if
$m\ge m_0(\alpha,\beta,\eta,\psi,k)$ is chosen large enough, then
the choice
$$ (H, F_mg )\quad \text{ with } \quad H=\supp (F_m g)\cup A_{\ep'}
$$ will give the desired  approximation of $(K, g)$.

\medskip

{\it Notation:} In this proof we shall write $B_1\lc B_2$ for two nonnegative quantities $B_1$, $B_2$ if  $B_1\le CB_2$ where $C$ may only depend on $\alpha$, $\beta$, $\gamma$, $\eps_1$, $k$, $d$ and $\ep$  and on the function $g$ (so $C$ will not depend on $\eta$ or $m$).
 We shall call such a $C$ an admissible constant.

\medskip

To show that $(H,F_mg\la)\in\fV^{\gamma,\varepsilon_1}$, we only need to verify \eqref{bounded-convolution-low} and \eqref{covering-property}. We postpone  \eqref{bounded-convolution-low} to a later part of the proof and now verify \eqref{covering-property}. By \eqref{eq:supp-F}
$$\text{supp}(F_mg)\subset \bigcup_{j} Q_j$$
where $Q_j$, $j=1,\cdots,(2m+1)^d\lfloor m^{k\beta}\rfloor$, are cubes with
side length $m^{-k-1}$. Thus $H=\text{supp}(F_mg)\cup A_{\ep'}$ can be covered by $$M=(2m+1)^d\lfloor m^{k\beta}\rfloor+(\#A_{\ep'})$$ cubes of side length ${m^{-k-1}}$. To verify \eqref{covering-property}, it now suffices to show
${m^{-k-1}}<\varepsilon_1 M^{-1/\gamma},$
which follows from
$3^{d} m^{k\beta+d}+(\#A)<{\varepsilon_1}^{\gamma} m^{k\gamma+\gamma}.$
Since $k>\tfrac{d-\gamma}{\gamma-\beta}$, the last inequality holds provided that $m$ is large enough.

We need to show that for sufficiently large  $m$
$$d_{\fW}\big((K,g\la),(H,F_m g\la)\big)<\epsilon.$$
Since $\text{supp}(F_mg)\subset\text{supp}(g)\subset K$,
we have $H=\text{supp}(F_mg)\cup A_{\ep'}\subset K$. Thus the Hausdorff distance of $H$ and $K$ satisfies
\Be\label{Hddist}d_{\fK}(H,K)\le \ep'.
\Ee

To handle the other components of $d_{\fW}$, we set
\Be \label{choiceofL}
L=10 \fn k d
\Ee
and
 we will use the fact that, since $g$ is smooth, there exists an admissible  constant $C>0$ such that
\begin{equation}\label{eq:g-approx}
\sum_{|r|_{_\infty}\ge m}|r|^L|\widehat g(r)|\le C m^{-(k+2)L}
\end{equation}
for all $m\ge 1$.
By the
periodicity of $F_m$, we have
\begin{align*}
|\widehat g(0)-\widehat{F_mg}(0)|
= \Big|\sum_{u\neq 0} \widehat {F_m}(-u)\widehat g(u) \Big|
\le \sum_{|u|_{_\infty}\ge m} |\widehat g(u)|\le C m^{-1}
\end{align*}
and hence
\Be\label{zeroFourier}
|\widehat g(0)-\widehat{F_mg}(0)|\le\ep'
\Ee
provided that $m$ is large enough.

For the nonzero Fourier coefficients we have,
\begin{align*}
&|\widehat{g}(r)-\widehat{F_mg}(r)|\,
= \,\Big|\sum_{u\neq r} \widehat {F_m}(r-u)\hat g(u)\Big |\\
&\le \sum_{|u|\le |r|/2} |\widehat {F_m}(r-u)\hat g(u)|
+\sum_{\substack{|u|> |r|/2\\ u\neq r}} |\widehat {F_m}(r-u)\hat g(u)|.
\end{align*}
By  \eqref{eq:fourier-decay-F},  this is estimated by
\begin{align*}
&\eta C_\psi \psi(|r|^{-1})2^{\alpha/2}|r|^{-\alpha/2} \sum_{|u|\le |r|/2} |\widehat g(u)|
+\eta \psi(1)\sum_{|u|> |r|/2} |\hat g(u)|\\
&\lc\big(\eta \psi(|r|^{-1})|r|^{-\alpha/2}+ \eta |r|^{-d}\big)\lc \eta \psi(|r|^{-1})|r|^{-\alpha/2}
\end{align*}
and this is $< \psi(|r|^{-1})|r|^{-\alpha/2} \ep'$
provided that $\eta>0$ is chosen  small enough. With this choice of
$\eta$ we have proved
\begin{equation}\label{eq:fourier-approx}
\sup_{r\in\mathbb Z^d\backslash\{0\}}\frac{
|r|^{\alpha/2}}{\psi(|r|^{-1})}|\widehat{g}(r)-\widehat{F_mg}(r)|< \ep'\,.
\end{equation}
if $\eta$ is sufficiently small and $m$ is sufficiently large.

It remains to show that \eqref{bounded-convolution-low} holds for
$\mu=F_m g\la$, i.e.
\Be\label{bd-conv-low}
\int_Q (F_m g)^{*n} (x)dx \le \psi(|Q|)|Q|^{n\alpha/d}, \quad 1\le n<d/\alpha
\Ee
and that, for $d/\alpha\le n\le \mathfrak{n}$
\begin{equation}\label{eq:convol-approx}
\sum_{i_1,\cdots,i_n}\big \|\big(\chi_{i_1}^{(n)}g\big)*\cdots*\big(\chi_{i_n}^{(n)}g\big)-\big(\chi_{i_1}^{(n)}F_m g\big)*\cdots*\big(\chi_{i_n}^{(n)}F_m g\big)
\big\|_{C^{\rho_n,\psi}}<\ep',
\end{equation}
provided that $\eta$ is small enough and $m$ is large enough.
Notice that
by the definition of the metric $d_\fW$ and
by \eqref{choiceofn} the corresponding terms for $n>\fn$ can be ignored.

\medskip

\noi{\it Proof of \eqref{bd-conv-low}.}
Following \cite{koerner2008} we write
$$P_m(x)=\sum_{|r|_{_\infty}\le m}\hat g(r) e^{2\pi i \inn rx}.$$
By \eqref{eq:g-approx} we have, for sufficiently large $m$
\Be \label{g-pm}\|g-P_m\|_{C^L}\le Cm^{-(k+2)L}\le 1.\ee
We first verify that for every $n=1,\dots,\fn$,
\begin{subequations}
\begin{gather}
\label{prelCLbd1}
\|g^{*n}-(P_m)^{*n}\|_{C^L}\le m^{-1},\\
\label{prelCLbd2}\|(F_m g)^{*n}-(F_mP_m)^{*n}\|_\infty\le m^{-1},
\end{gather}
provided that  $m$ is chosen large enough.
\end{subequations}
To see this we write
\begin{align*}
&g^{*n}-(P_m)^{*n}
=\big((g-P_m)+P_m\big)^{*n}-(P_m)^{*n}\\
&=(g-P_m)^{*n}+\sum_{\nu=1}^{n-1}\binom{n}{\nu}
(g-P_m)^{*(n-\nu)}*(P_m)^{*\nu}.\end{align*}
Therefore, using $\binom{n}{\nu}= \tfrac{n}{n-\nu} \binom{n-1}{\nu}$ for $1\le\nu\le n-1$ and \eqref{g-pm},
\begin{align*}
&\|g^{*n}-(P_m)^{*n}\|_{C^{L}}\le  \|g-P_m\|_{C^{L}}\sum_{\nu=0}^{n-1}
\binom{n}{\nu}\|P_m\|_\infty^\nu \\
&\le \|g-P_m\|_{C^{L}}\big(1+\|P_m\|_\infty\big)^{n-1}n
\lc m^{-2}n \big(2+\|g\|_\infty\big)^{n-1}
\end{align*}
and this gives \eqref{prelCLbd1}
provided that $m$ is large enough.

By \eqref{eq:bdd-derivative-F} and the first estimate in \eqref{g-pm} we have
\begin{align*}
&\|F_m(P_m-g)\|_{C^{L}}
\lc\|F_m\|_{C^{L}}\|P_m-g\|_{C^{L}}\\
&\qquad\lc m^{kd+(k+1)L} m^{-(k+2)L} \le Cm^{kd-L}\le 1
\end{align*}
for sufficiently large  $m$
The same argument as above then gives
\begin{align*}
&\|(F_m P_m)^{*n}-(F_m g)^{*n}\|_{C^{L}}
\le  \|F_m(P_m-g)\|_{C^{L}} n\big(1+\|F_m g\|_\infty\big)^{n-1}n\\
&\lc m^{kd-L} \big(1+m^{kd}\|g\|_\infty\big)^{n-1}n \lc
m^{nkd-L} \big(1+\|g\|_\infty\big)^{n-1}n
\end{align*} and this gives
and this gives \eqref{prelCLbd2}
provided that $m$ is large enough.

As a consequence
of part (ii) of Lemma \ref{lemma:periodprop} we have
\begin{equation}\label{eq:separate}
(F_m P_m)^{*n}=(F_m)^{*n}(P_m)^{*n}.
\end{equation}
Now for fixed $n<d/\alpha$ and a cube $Q$,  we have by \eqref{eq:separate}
and \eqref{prelCLbd2}
\begin{align*}
\int_Q(F_m g)^{*n}(x)dx
&\le \Big|\int_Q(F_m P_m)^{*n}(x)dx\Big| + \Big|\int_Q \big((F_m g)^{*n}(x)-(F_m P_m)^{*n}(x)\big)dx\Big|\\
&\le \Big|\int_Q (F_m)^{*n}(x) (P_m)^{*n}(x)dx\Big| + m^{-1}|Q|\\
&\le \Big|\int_Q (F_m)^{*n}(x) (P_m)^{*n}(x)dx\Big| + Cm^{-1}\psi(|Q|)
|Q|^{n\alpha/d}\\
&\le \Big|\int_Q(F_m)^{*n}(x) (P_m)^{*n}(x)dx\Big| + \frac{c}{2}
\psi(|Q|)|Q|^{n\alpha/d}
\end{align*}
for sufficiently large $m$.
Thus, in order to finish the proof of \eqref{bd-conv-low} we must show
\Be\label{bd-conv-low-mod}
\int_Q (F_m)^{*n}(x)  (P_m)^{*n}(x)dx\le  (1-\frac{c}{2})\psi(|Q|)|Q|^{n\alpha/d}
\Ee
If the side length of $Q$ is $\le 2/\sqrt{m}$, then
\begin{align*}
&\Big|\int_Q (F_m)^{*n}(x) (P_m)^{*n}(x)dx\Big|
\le \|(P_m)^{*n}\|_\infty \int_Q(F_m)^{*n} (x)dx
\\&\le (1+\|g^{*n}\|_\infty) \eta  \psi(|Q|)|Q|^{n\alpha/d}
\le (1-\frac{c}{2})\psi(|Q|)|Q|^{n\alpha/d}
\end{align*}
where in the last inequality $\eta$ is chosen sufficiently large (the second inequality follows from
\eqref{eq:regularity-F-1}).

If the side length of $Q$ is $> 2/\sqrt{m}$, then $Q$ can be split into rectangles $R$ of side lengths between $1/\sqrt m$ and $2/\sqrt{m}$. Writing
$$a_R=\dashint_R g^{*n}(x)dx, \qquad b_R=\dashint_R (P_m)^{*n}(x)dx,$$
we then have
$$\|(P_m)^{*n}-b_R\|_{L^\infty(R)}\lc m^{-1/2}\|(P_m)^{*n}\|_{C^1}\lc
m^{-1/2} $$ and
$$|b_R-a_R|\le \| g^{*n}-(P_m)^{*n}\|_\infty\le m^{-1},$$
by \eqref{prelCLbd1}.  Now
\begin{align*}
&\Big|\int_Q (F_m)^{*n}(x) (P_m)^{*n}(x)dx\Big| \,\le \Big|\sum_R \int_R (F_m)^{*n} (x)a_R dx\Big|\\
&\ +\Big|\sum_R \int_R (F_m)^{*n} (x)(b_R-a_R)dx\Big|+\Big|\sum_R \int_R (F_m)^{*n}(x)\big(b_R-(P_m)^{*n}(x) \big) dx\Big|\\
&\le \sum_R a_R (1+\eta)|R|+\Big(\frac{1}{m}+\frac{C}{\sqrt m}\Big)\sum_R \int_R (F_m)^{*n}(x) dx\\
&\le (1+\eta)\int_Q g^{*n}(x)dx+\frac{C'}{\sqrt m} \int_Q (F_m)^{*n}(x) dx
\end{align*}
where $C'$ is admissible. By \eqref{eq:regularity-F-2} and
\eqref{bounded-convolution-low-c} the
 last expression
is less than or equal to
\begin{align*}&(1+\eta)(1-c)\psi(|Q|)|Q|^{n\alpha/d}
+\frac{C'}{\sqrt m}(1+\eta)|Q|\\
&\le \Big((1-\frac{3}{4}c)
+\frac{C''}{\sqrt m}\Big)\psi(|Q|)|Q|^{n\alpha/d}\le (1-\frac{c}{2})\psi(|Q|)|Q|^{n\alpha/d}
\end{align*}
provided that $\eta$ is small enough and $m$ is large enough.

In either case we have verified
\eqref{bd-conv-low-mod},
and this concludes the proof of
\eqref{bd-conv-low}.

\medskip

\noi {\it Proof of  \eqref{eq:convol-approx}.} Fix $n$ with $d/\alpha\le n\le
 \fn$
and $i=(\imath_1,\cdots,\imath_n) \in (\fI_n)^n$.
 Write
$$g_j=\chi_{\imath_j}^{(n)}g,$$ for $j=1,\cdots, n$, and
$$P_{j,m}(x)=\sum_{|r|_\infty\le m} \widehat {g_j}(r) e^{2\pi i\inn{r}{x}}.$$
\eqref{eq:convol-approx} reduces to estimating
\begin{align*}
&\|g_1*\cdots*g_n-(F_m g_1)*\cdots*(F_m g_n)\|_{C^{\rho_n,\psi}}\\
&\le \|g_1*\cdots*g_n-P_{1,m}*\cdots*P_{n,m}\|_{C^{\rho_n,\psi}}\\
& +\|P_{1,m}*\cdots*P_{n,m}-(F_m P_{1,m})*\cdots*(F_m P_{n,m})\|_{C^{\rho_n,\psi}}\\
&+\|(F_m P_{1,m})*\cdots*(F_m P_{n,m})-(F_m g_1)*\cdots*(F_m g_n)\|_{C^{\rho_n,\psi}}.
\end{align*}
Arguing as before (\cf. \eqref{eq:g-approx}), we have for sufficiently large $m$
$$\|P_{j,m}-g_j\|_{C^L}\le C m^{-(k+2)L}\le 1$$
and
$$\|F_m(P_{m,j}-g_j)\|_{C^L}\le Cm^{kd-L}\le 1.$$
Using the continuous embedding $C^L \rightsquigarrow C^{\rho_n,\psi}$ we get therefore, for sufficiently large $m$,
\begin{align*}
 \|g_1*\cdots*g_n-P_{1,m}*\cdots*P_{n,m}\|_{C^{\rho_n,\psi}}\\
\le C m^{-(k+2)L} n \prod_{j=1}^n (1+\|g_j\|_\infty)\le m^{-1}
\end{align*}
and
\begin{align*}
&|(F_m P_{1,m})*\cdots*(F_m P_{n,m})-(F_m g_1)*\cdots*(F_m g_n)\|_{C^{\rho_n,\psi}}\\
&\le C m^{nkd-L} n \prod_{j=1}^n (1+\|g_j\|_\infty) \le m^{-1}.
\end{align*}

On the other hand, using Lemma \ref{lemma:periodprop}, (ii),
again we have
$$(F_m P_{1,m})*\cdots*(F_m P_{n,m})=(F_m)^{*n}(P_{1,m}*\cdots*P_{n,m}).$$
Thus, by \eqref{eq:bdd-convol-F}
\begin{align*}
&\|P_{1,m}*\cdots*P_{n,m}-(F_m P_{1,m})*\cdots*(F_m P_{n,m})\|_{C^{\rho_n,\psi}}\\
&= C \|(1-F_m^{*n})(P_{1,m}*\cdots*P_{n,m})\|_{C^{\rho_n,\psi}}\\
&\lc \|1-F_m^{*n}\|_{C^{\rho_n,\psi}}\|P_{1,m}*\cdots*P_{n,m}\|_{C^{\rho_n,\psi}}\\
&\lc\eta\|P_{1,m}*\cdots*P_{n,m}\|_{C^{L}}\lc  \eta(1+\|g_1*\cdots*g_n\|_{C^{L}})\lc \eta
\end{align*}
provided that $m$ is sufficiently large.

Combining the above estimates, we get
$$\|g_1*\cdots*g_n-(F_m g_1)*\cdots*(F_m g_n)\|_{C^{\rho_n,\psi}}\lc  m^{-1}+\eta.$$
This guarantees \eqref{eq:convol-approx} if $\eta$ is chosen sufficiently small and $m$ is chosen sufficiently large.
This completes the proof of Lemma \ref{lemma:approx}.
\end{proof}

\subsection{\it Conclusion of the proof of
Theorem \ref{intro:theorem-C} } The result is about measures on $\bbR^d$ rather than $\bbT^d$.
We use that every measure on $\bbT^d$  which is supported on a cube of sidelength $<1$ can be identified with a measure that is supported on a cube of diameter $<1$
in $\bbR^d$. We take a measure  $\mu$ as in Corollary \ref{quasi-all}. After multiplying  it with a suitable $C^\infty_c$ function we may assume that it is supported  on a cube of diameter $<1$. For each $n$ we may decompose $\mu$ using the partition of unity \eqref{eq:partofunity}. The regularity properties (iii) and (iv) in Theorem \ref{intro:theorem-C} follow immediately from \eqref{bounded-convolution} and \eqref{bounded-convolution-low}.  The compact support of $\mu$ and the
 decay property \eqref{FTd-decay}  on  $ \bbZ^d$ imply the decay property in (ii).
 This is a standard argument (see e.g. \cite{kahane}, p.252, with slightly different notation). \qed

\subsection{\it Optimality of H\"{o}lder continuity}
\label{optimalityHoelder}
Following the argument in \cite{koerner2008}, we show that the H\"{o}lder continuity obtained in Theorem \ref{intro:theorem-C} is best possible.

\begin{proposition}\label{prop:lipschitz-best}
Let $\mu$ be a Borel probability measure on $\mathbb R^d$ supported on a compact set of Hausdorff dimension $0<\alpha<d$. Suppose $\mu^{*n}\in C^{\lambda}(\mathbb R^d)$
where $n\in\mathbb N, n\ge 2$ and $0\le\lambda<\infty$. Then $\lambda\le \frac{n\alpha-d}{2}.$
\end{proposition}

\begin{proof}
Define by $\cE_\ga(\mu)
=\iint|x-y|^{-\ga}d\mu(x)d\mu(y)=
c\int|\widehat \mu(\xi)|^2|\xi|^{\gamma-d} d\xi$ the $\gamma$-dimensional energy of $\mu$. Recall
from \cite[p.62]{wolff} that  the Hausdorff dimension of $E$  is equal to
the supremum over all $\ga$ for which there is a probability measure $\nu$ supported on $E$ with
$\cE_\ga(\nu)<\infty$. Thus it suffices to show that
$\cE_\ga(\mu)$ is finite  for
$\gamma<(d+2\lambda)/n$.

Since $\mu^{*n}$ is compactly supported it also belongs to the Besov-space
$B^2_{\la,\infty}$ and thus, by Plancherel,  we have, for $R>1$,
$$\int_{|\xi|\approx R}|\hat \mu(\xi)|^{2n} d\xi \lc R^{-2\la}$$
Now let
$0<\gamma<d$. By  H\"{o}lder's inequality,
$$
\int_{|\xi|\approx R}\frac{|\widehat \mu(\xi)|^2}{|\xi|^{d-\gamma}} d\xi
\lesssim R^{\ga-d}\Big(\int_{A_R}|\hat \mu(\xi)|^{2n} d\xi\Big)^{1/n}R^{d(1-1/n)}
\lesssim R^{\ga-d/n-2\lambda/n}.
$$
Letting  $R=2^j$, $j=0,1,\cdots$, we see that
$\cE_\gamma(\mu)$
is finite
if $\gamma<(d+2\lambda)/n$  and the proof is complete.
\end{proof}

\section{Random sparse subsets}\label{sec:point-masses}
The purpose of this section is to establish a more quantitative version of
Proposition \ref{prop:point-masses}.

\subsection{\it Assumptions and Notations:} In  this chapter
 $x_1, x_2,\dots$  will be  independent random variables uniformly distributed
 on $\Gamma_N^d$.  That is,  for  any $m\in \bbN$ and subsets $A_1, \dots, A_m$  of $\Gamma_N^d$
the probability of the event that
$x_\nu\in A_\nu$ for $\nu=1,\dots,m$ is equal to
$N^{-dm} \prod_{\nu=1}^m\card (A_\nu\cap\Gamma_N^d)$.
We  denote by $\sF_0$ the trivial $\sigma$-algebra and by
 $\sF_j$ the $\sigma$-algebra generated by the (inverse images) of the random variables $x_1,\dots, x_j$.


Given random Dirac masses $\delta_{x_\nu}$, $\nu=1,\dots,m$  we define the random measures $\mu_m$ and $\sigma_m$ by $\sigma_0=\mu_0=0$,
$$\sigma_m =\sum_{\nu=1}^m \delta_{x_\nu}, \qquad \mu_m = m^{-1}\sigma_m, \quad
\text{  $m=1,2,\dots$.}
$$

\subsection{\it A Fourier decay estimate}
The  Fourier transform $\widehat \mu$ is defined on $\bbZ_N^d$ or,  after scaling,  on $\Gamma_N^d$ and we have
$$\widehat \mu_m(Nu) =\frac{1}{m} \sum_{j=1}^m e^{-2\pi i N\inn{u}{x_j}}, \quad u\in \Gamma_N^d.$$

\begin{lemma}\label{fourier-decay-discrete}
Let $h\ge 1$.
The event
\begin{equation}\label{fourier-decay-discrete-equation}
\Big\{\max_{u\in \Ga_N^d\backslash\{0\}}|\widehat\mu_m(Nu)|\le \frac{4\log^{1/2}(8N^{d+h})}{m^{1/2}}\Big\}
\end{equation}
has probability at least $1-N^{-h}$.
\end{lemma}
\begin{proof}
The proof is essentially the same as in the classical paper by Erd\H os and R\'{e}nyi \cite{erdos-renyi}.  Fix $u\in \Gamma_N^d\backslash\{0\}$, and consider the random variables
$X_\nu=e^{-2\pi i N\inn{u}{ x_\nu}}.$
Then $X_\nu$, $\nu=1,\dots,m$ are independent with $|X_\nu|\le 1$ and $\mathbb EX_j=0$. Thus by Bernstein's inequality (see e.g. Corollary \ref{lemma:bernstein}), for all $t>0$
$$\mathbb P(|\widehat\mu(Nu)|\ge t)\le 4e^{-mt^2/4}.$$
Setting $t=2m^{-1/2}\log^{1/2}(4N^{d+h})$
we get
$\mathbb P\{|\widehat\mu(Nu)|\ge t\}\le N^{-d-h}$.
Allowing $u \in \Gamma_N^d$ to vary, we see that
$\mathbb P\big\{\text{\eqref{fourier-decay-discrete-equation} fails}\big\}\le N^{-h}$.
\end{proof}

\subsection{\it Regularity of self convolutions}
We begin with a few elementary observations. Let
\Be\label{eq:Deljdef}
\Delta_{j,\ell}= \sigma_j^{*\ell}-\sigma_{j-1}^{*\ell}\,.
\Ee so that
\Be \label{eq:telescope} \sigma_m^{*\ell}=\sum_{j=1}^m \Delta_{j,\ell}\,.\Ee

\begin{lemma}\label{lemma:elem}
(i) For $j\ge 1$,
$\Delta_{j,\ell}$
 is a positive measure, and we have, for $\ell\ge 2$,
\begin{subequations}
\begin{align}
\Delta_{j,\ell}
&= \delta_{\ell x_j}+
\sum_{k=1}^{\ell-1}\binom{\ell}{k}\delta_{(\ell-k)x_j}*\sigma_{j-1}^{*k}\,
 \label{binomial-formula-1}\\
&=\delta_{\ell x_j}+
\sum_{k=1}^{\ell-1}\binom{\ell}{k}\sum_{1\le \nu_1,\dots,\nu_k\le j-1}
\delta_{(\ell-k)x_j+x_{\nu_1}+\dots+x_{\nu_k}}
 \,. \label{binomial-formula-2}
\end{align}
\end{subequations}

(ii) Assume that $\gcd (\ell!, N)=1$.
Let $m\ge 2$ and let $Q$ be a cube of  sidelength $\ge N^{-1}$.
Then for $j_1<\dots<j_K$
$$\bbP\big\{\Delta_{j_1,\ell}(Q)\neq 0, \dots, \Delta_{j_K,\ell}(Q)\neq 0\big\}\le
(2^{d+1}|Q| m^{\ell-1})^K.$$
In particular, for each $u\in \Gamma_N^d$
$$\bbP\big\{\Delta_{j_1,\ell}(\{u\})\neq 0, \dots, \Delta_{j_K,\ell}(\{u\})\neq 0\big\}\le
(2 N^{-d} m^{\ell-1})^K.$$
(iii)
 Assume that  $\gcd (\ell!, N)=1$.
For $j=0,\dots, m-1$ let $\cE_j$ be a given event in $\sF_j$. Let
\Be\label{eq:Yjdef}
Y_j=\begin{cases}
\sigma_j^{*\ell}-\sigma_{j-1}^{*\ell}- N^{-d}(j^\ell-(j-1)^\ell)&\text{ on }
\cE_{j-1},
\\0 &\text{ on }  \cE_{j-1}^\complement.
\end{cases}
\Ee
Then $\mathbb E\big[Y_j|\mathscr F_{j-1}\big]=0$. Let $W_0=0$ and $W_j=\sum_{\nu=1}^j  Y_\nu$, for $j=1,\dots, m$. Then $\{W_j\}_{j=0}^m$ is a martingale adapted to the filtration $\{\sF_j\}_{j=0}^m$.
\end{lemma}

\begin{proof} Part (i) follows immediately from the binomial formula.
 For part (ii)
 note that by the assumption $\gcd (\ell!, N)=1$ the random variables
$(\ell-k)x_j$, $1\le k\le  \ell$, are uniformly distributed.
Observe that for any fixed $a$ the probability of the event
$\{(\ell-k)x_d-a\in Q\}$  is at most $2^d|Q|$.
Thus the probability of the event that
$(\ell-k)x_d-a\in Q$  for some choice of $a=x_{\nu_1}+\dots+x_{\nu_k}$,
$1\le \nu_1,\dots,\nu_k\le j-1$,
 does not exceed
$2^{d}|Q| (j-1)^{k-1}.$  Hence
$\bbP\big\{\Delta_{j,\ell}(Q)\neq 0\} \le
 2^d|Q|\sum_{\ka=0}^{\ell-1}{m^\ka}\le
 2^{d+1}|Q|m^{\ell-1}\,.$ Now the assertion in  part (ii) follows.
The second assertion in (ii) is proved similarly.

For (iii), clearly $\{W_j\}_{j=0}^m$ is  adapted to the filtration $\{\sF_j\}_{j=0}^m$.
By assumption the random variable $qx_j$ is uniformly distributed on $\Gamma_N^d$, for $1\le q\le \ell$.
Given fixed  $x_1,\cdots,x_{j-1}$,
then by \eqref{binomial-formula-2}
\begin{align*}
&\mathbb E\big[\sigma_j^{*\ell}(\{u\})-\sigma_{j-1}^{*\ell}(\{u\})|x_1,\cdots,x_{j-1}\big] \\
&=N^{-d}\sum_{q=0}^{\ell-1} \binom{\ell}{q}(j-1)^q=N^{-d}(j^\ell-(j-1)^\ell).
\end{align*}
Since $\cE_{j-1}\in \sF_{j-1}$ we get
 $\mathbb E\big[Y_j\bbone_{\cE_{j-1}}|\sF_{j-1}\big]=0$ in this case. On $\cE_{j-1}^\complement$ we have  $Y_j=0$, which also implies $\mathbb E\big[Y_j \bbone_{E_{j-1}^\complement}|\sF_{j-1}\big]=0$. Hence $\mathbb E\big[Y_j|\mathscr F_{j-1}\big]=0$ and this shows $\{W_j\}_{j=0}^m$ is a martingale.
\end{proof}

We shall use (a small variant of) an elementary inequality from K\"orner's paper
(\cite[Lemma~11]{koerner2008})
which is  useful for the estimation of sums of
independent Bernoulli  variables.
\begin{lemma}[\cite{koerner2008}]\label{small-summation}
Let $0<p<1$, $m\ge 2$ and $2mp\le M\le m$. Then
$$\sum_{k=M}^m\binom{m}{k} p^k \le
\frac{2(mp)^M}{M!}.$$
In particular, if
 $mp\le 1$ and if $Y_1,\cdots,Y_m$ are independent random variables with
$\mathbb P\{Y_j=1\}=p$, $\mathbb P\{Y_j=0\}=1-p$ then
$\mathbb P\{\sum_{j=1}^m Y_j\ge M\}\le \frac{2(mp)^M}{M!}.$
\end{lemma}

\begin{proof}
Set $u_k=\binom{m}{k}p^k,$
then  $u_{k+1}/u_k=\frac{m-k}{k+1}p\le\frac{mp}{k+1}\le\frac{1}{2}$ for $k\ge M$ and thus
the sum is estimated by
$\sum_{k\ge M}u_k \le 2u_M\le \frac{2(mp)^M}{M!}.$
The second assertion follows since
$\mathbb P\{\sum_{j=1}^m Y_j\ge M\}=\sum_{k=M}^m\mathbb P\{\sum_{j=1}^m Y_j=k\}\le\sum_{k=M}^m u_k.$
\end{proof}

For $\ell=0,1,2,\dots$, $0<\eps<d$, and $h\in \bbN$ define  recursively positive numbers
$M(\ell, \eps, h)$ by
\begin{subequations}\label{Mconstantsdef}
\Be\label{eq:Mlepsh}
\begin{aligned}
M(0, \eps, h)&=1\\
M(\ell, \eps, h) &= U(\eps,h)\ka(\ell,h), \quad \ell\ge 1,
\end{aligned} \Ee
where
\begin{align}
\label{eq:Uepsldef}
U(\eps,h)&:=\max \{\lfloor e^{d+2} \rfloor, \lceil \eps^{-1}(2d+h+1)\rceil \},
\\
\label{eq:kappadef}\ka(\ell, h)&:=
\sum_{q=0}^{\ell-1}\binom {\ell}{q} M(q,
d(1- q/\ell), h+1)\,.
\end{align}
\end{subequations}
The growth of these constants as functions of $\ell$ and $h$ is irrelevant for our purposes. For the sake of  completeness we give an upper bound.
\begin{lemma}\label{le:growth}
Let $\ell \in \bbN\cup \{0\}$, $0<\eps<d$, and $h\in \bbN$.
The numbers  defined in  \eqref{Mconstantsdef} satisfy
$$M(\ell,\eps, h)\le \eps^{-1}(e^{d+3}\ell^2 (h+\ell))^\ell.$$
\end{lemma}
\begin{proof} We argue  by induction, with the case $\ell=0$ being trivial.
For the induction step we use
$\binom {\ell}{q}= \frac{\ell}{\ell-q}\binom{\ell-1}{q}$ and
estimate\begin{align}\kappa(\ell,h)&\le 1+ \sum_{q=1}^{\ell-1}
\frac{\ell}{\ell-q} \binom {\ell-1}{q} \frac{(e^{d+3}q^2 (h+1+q))^q}{d(1-\frac q\ell)}
\notag\\&\le \ell^2 \sum_{q=0}^{\ell-1}\binom {\ell-1}{q} (e^{d+3}(\ell-1)^2 (h+\ell))^q
\label{eq:kastarline}\\&\le  \ell^2\big( e^{d+3} (\ell-1)^2(h+\ell)+1\big)^{\ell-1}
\notag\end{align}where in the last line we have used $(1+x)^{1/x}\le e^x$ for $0<x<1$. Thus
\Be\label{eq:kaest}\kappa(\ell,h)\le e^{1/2} \ell^2 (e^{d+3}\ell^2 (h+\ell))^{\ell-1}.\Ee
Now one checks that
$U(\eps, h)\le e^{d+2} h \eps^{-1}$ and \eqref{eq:kaest} yields
 for $\ell\ge 1$\begin{equation*}M(\ell, \eps, h)\le e^{d+2}h\eps^{-1} \kappa(\ell,h) \le
\eps^{-1}( e^{d+3}\ell^2(h+\ell))^{\ell}.\qedhere\end{equation*} \end{proof}

\begin{lemma}\label{le:ball}
Let $\ell \in \bbN\cup \{0\}$,
$0<\eps<d$, and $h\in \bbN$.
 Let $M(\ell, \eps, h)$ be as in \eqref{eq:Mlepsh}.
Let $N$ be an integer such that  $N> 2\ell$ and $\gcd(N,\ell!)=1$.
Let $E_m(\ell,\eps,h)$ be the event that $$\sigma_m^{*\ell}(Q)\le M(\ell,\eps,h)$$
holds for all cubes of measure at most $m^{-\ell}N^{-\eps}$, and let
$E(\ell,\eps, h)$ be the intersection of the  $E_m(\ell,\eps,h)$ where
$m\le N^{\frac{d-\eps}{\ell}}$. Then
$E(\ell,\eps, h)$
has probability at least $1-N^{-h}$.
\end{lemma}


\begin{proof}
We argue again by induction on $\ell$.
 When $\ell=0$, $\sigma^{*0}=\delta_0$ and the statements clearly holds
with $M(0,\eps,h)=1$, for $\eps\ge 0$ and $h\in \bbN$.
Assume that the statements hold for $0,1,\cdots,\ell-1$; we prove that it also holds for $\ell$. Let
\Be\label{eq:cFdef}F\equiv F(\ell-1,h)= \bigcap_{q=1}^{\ell-1}
 E(q,\eps_{q,\ell},h+1),
\quad \text{ with }\eps_{q,\ell}=d(1-\frac q\ell)\,.
\Ee
By the induction hypothesis,  the event $F^\complement$ has probability at most
$\ell N^{-h-1}\le \frac 12 N^{-h}$ since we assume $N>2\ell$.
We now proceed to estimate the probability of $E(\ell,\eps,h)^\complement\cap F$.

Fix $m\le N^{\frac{d-\eps}\ell}$, and
fix a cube $Q$, with $N^{-d}\le |Q|\le N^{-\eps}m^{-\ell}$.
Notice that
$d/\ell = (d-\eps_{q,\ell})/q$.
Therefore, if $\kappa(\ell,h)$ is as in \eqref{eq:kappadef}
we see, using  \eqref{binomial-formula-1}, that
$\Delta_{j,\ell}(Q) \le \ka(\ell,h)$ holds on $ F$, for $j=1,\cdots,m.$
Now let $U\ge 2^{d+2}$ be an integer
and let  $\cA_{U,m}^Q$ be the event that
\begin{equation}\label{eq:Ukappa}
\sum_{j=1}^m \Delta_{j,\ell}(Q)\ge U \kappa(\ell,h).
\end{equation}

Now by \eqref{eq:telescope} and \eqref{eq:Mlepsh} the event $E_m(\ell, \eps,h)^\complement$
is contained in the union over the $\cA^Q_{U(\eps,h), m}$ when $Q$ ranges over the cubes with measure at most $N^{-\eps}m^{-\ell}$.
Let $\fQ$ be the collection
 of all cubes of measure  $N^{-\eps}m^{-\ell}$,
which have corners in $\Gamma_N^d$. Then $\#(\fQ)\le (2N)^d$. Notice that every cube of measure less than $N^{-\eps}m^{-\ell}$ is contained in at most $3^d$ cubes in $\fQ$.  Hence
\Be\label{eq:comparisonE-A}
\bbP( E_m(\ell,\eps,h)^\complement\cap F) \le (6N)^d
\max_{Q\in \fQ} \bbP(\cA_{U(\eps,h),m}^Q\cap F).
\Ee
Now in order  to estimate
$\bbP(\cA_{U,m}^Q\cap F) $ we observe that if \eqref{eq:Ukappa}
holds on $F$ then there are at least $U$ indices $j$ with $\Delta_{j,\ell}(Q)\neq 0$ thus we may assume $m\ge U$.
Now we see
 from  Lemma \ref{lemma:elem}, (ii),  that for $U\le k\le m$ and
 for any choice of  indices $1\le j_1<\cdots<j_{k}\le m$,
that
$$\bbP\{\Delta_{j_\nu,\ell}(Q)\neq 0,\,\nu=1,\dots, k\} \le (2^{d+1}|Q|m^{\ell-1})^k.$$
Thus
$$\mathbb P\big(\cA_{U,m}^Q\cap F \big )\le
\sum_{k=U}^{m} \binom{m}{k}\big(2^{d+1}|Q| m^{\ell-1}\big)^{k}.$$
Now let  $p:=2^{d+1}|Q| m^{\ell-1}$.
Since
$|Q|<m^{-\ell}$ we  have $mp \le 2^{d+1}$. Since we assume $U\ge 2^{d+2}$ we get from
Lemma \ref{small-summation},
$$\sum_{k=U}^{m} \binom{m}{k}\big(2^{d+1}|Q| m^{\ell-1}\big)^{k}\le
\frac{2(mp)^U}{U!} \le \frac{2(2^{d+1}N^{-\eps})^U}{U!}.$$
Thus we get from \eqref{eq:comparisonE-A}
\Be\label{Emestimate}\bbP( E_m(\ell,\eps,h)^\complement\cap F )\le (6N)^d\frac{2(2^{d+1}N^{-\eps})^{U(\eps,h)}}{U(\eps,h)!}.\Ee
It is not difficult to check  that
$$\frac {6^d \cdot 2(2^{d+1})^U}{U!} \le 1 \text{ for } U> e^{d+2}-1\,;$$
this can be verified by taking logarithms and replacing $\log U$ with the smaller constant $\int_1^{U-1}\ln(t)dt$. Since in addition
$U=U(\eps,h)\ge \lceil\tfrac{2d+h+1}{\eps}\rceil$ then we get
$N^{d-\eps U}\le \frac 12 N^{-d-h}$  and thus
$$\mathbb P(E_m(\ell,\eps,h)^\complement\cap F)  \le \frac 12 N^{-h-d}.$$

We have already remarked that  $\mathbb P(F^\complement)< \frac 12 N^{-h}$.
Thus,
\begin{align*}
\mathbb P\big(E(\ell,\eps, h)^\complement\big)
\le \mathbb P(F^\complement)+\sum_{m\le N^{\frac{d-\eps}\ell}}
\mathbb P\big(E_m(\ell,\eps,h)^\complement\cap F\big)
\le N^{-h}.
\end{align*}
 This completes the proof.
\end{proof}

\begin{lemma}\label{le:ballvariant}
Let $\ell \in \bbN$,
$0<\beta\le d/\ell$, and $h\in \bbN$.
Let $N$ be an integer such that  $N>\max\{ 2\ell, e^{e^e}\}$ and $\gcd(N,\ell!)=1$.
Let $\cE_m(\ell,\beta,h)$ denote  the event that
$$\sigma_m^{*\ell}(Q)\le
(\beta\ell)^{-1}\big(10^{d+1}\ell^2(\ell+h)\big)^\ell
\frac{\log N}{\log\log N}$$
holds for all cubes of measure at most $N^{-\beta\ell}$, and let
$$\cE(\ell,\beta, h)=\bigcap_{m\le N^\beta} \cE_m(\ell, \beta,h).$$ Then
$\cE(\ell,\beta, h)$
has probability at least $1-N^{-h}$.
\end{lemma}

\begin{proof} Let
\Be
\label{eq:tkappadef}\widetilde \ka(\ell, \beta, h):=
\sum_{q=0}^{\ell-1}\binom {\ell}{q} M(q,
\beta(\ell- q), h+1)\,
\Ee and let $V\ge e^{2d+8}h$ be a positive  integer. Let
$\widetilde E_m(\ell, h,V)$
denote  the event that
$$\sigma_m^{*\ell}(Q)\le \widetilde \kappa(\ell,\beta, h) V \frac{\log N}{\log\log N}$$
holds true for all cubes with  measure at most $N^{-\ell\beta}$.
We shall show that for sufficiently large $V$ the complement of  this event has small probability.

We condition on the event
\begin{align}\label{eq:tcFdef}\widetilde F&= \bigcap_{q=1}^{\ell-1}
 \cE(q,\beta(\ell-q),h+1)
 \\
 &=\Big\{\sigma_m^{*q}(Q)\le M(q,\beta(\ell-q), h+1)
  \notag
 \\&\qquad\qquad \forall Q \text{ with } |Q|\le m^{-q}N^{-\beta(\ell-q)}, \,\,
 1\le m\le N^{\frac{d-\beta(\ell-q)}{q}}\Big\}.
 \notag
 \end{align}
By Lemma \ref{le:ball},
$$\bbP(\widetilde F^\complement)\le \ell N^{-h-1}\le \frac 12 N^{-h}.$$

We shall now estimate $\bbP(\widetilde {E}_m(\ell, h,V)^\complement\cap \widetilde F)$.
The assumptions $m\le N^\beta$, $|Q|\le N^{-\ell \beta}$ with $\beta\le d/\ell$
 imply  for $q\le \ell-1$ that
$m\le N^{\frac{d-\beta(\ell-q)}{q}}$  (since $d-\beta\ell\ge 0$) and
$|Q|\le m^{-q}N^{-(\ell-q)\beta }$.  Thus we can use \eqref{binomial-formula-1}
to see that $\Delta_{j,\ell}(Q)\le \widetilde \ka(\ell,\beta,h)$ on $\widetilde F$, for $j=1,\dots,m$.

Let
$\widetilde \cA_{V,m}^Q$ be the event that
\Be\label{eq:AvmQ-event}
\sum_{j=1}^m\Delta_{j,\ell}(Q)\ge V_N \widetilde \kappa(\ell, \beta, h), \text{ where } V_N=\Big\lfloor V\frac{\log N}{\log\log N}\Big\rfloor\,.
\Ee
Let $\widetilde \cA_{V,m}$ be the event that
\eqref{eq:AvmQ-event} holds for some cube with measure at most $N^{-\ell\beta}$.
Arguing as in the proof of Lemma \ref{le:ball} we find that
$$\bbP(\widetilde \cA_{V,m} \cap \widetilde F)\le 2\cdot  (6N)^d \frac{ 2^{(d+1)V_N}}{V_N!}$$
\begin{subequations} \label{VNbounds} We need to verify that
\Be\label{eq:logbd-a}
2\cdot  (6N)^d \frac{ 2^{(d+1)V_N}}{V_N!}\le N^{-d-1-h}
\Ee for $V\ge e^{2d+8} h$ and $N>e^{e^e}$.
We take logarithms and replace  $\log V_N!$ with the lower bound
$\int_1^{V_N-1}\log t\, dt=(V_N-1)\log(V_N-1)-V_N+2$. Then \eqref{eq:logbd-a} follows from
\begin{multline}\label{eq:logbd-b}
\log 2 +
d\log 6 + V_N(1+(d+1)\log 2)-2- (V_N-1)\log(V_N-1)\\<-(h+1+d)\log N.
\end{multline}
Since by  assumption  $V\ge  e^{2d+10}$ and $N>e^{e^e}$    crude estimates
show that \eqref{eq:logbd-b}
is implied by
\Be \label{eq:logbd-c}\frac{V_N}2\log(V_N-1)
\ge (d+h+1)\log N.\Ee
For $N\ge e^{e^e}$ we have $\log\log\log N \le \frac 12\log\log N$ and therefore
$\log(V_N-1)\ge \frac 12 \log\log N$. Thus
\eqref{eq:logbd-c} is implied by
$V \ge 4(h+2+d)$
which holds since we assume  $V\ge  e^{2d+8} h$
and  $N\ge e^{e^e}$. Thus \eqref{eq:logbd-a} holds.
\end{subequations}
We thus get $\bbP(\widetilde {E}_m (\ell,h,V)^\complement\cap \widetilde F)\lc N^{-d-h-1}$ and hence
\begin{align*}\bbP(\cup_{m\le N^\beta}\widetilde {E}_m (\ell,h,V)^\complement)&\lc \bbP(\widetilde F^\complement)
+\sum_{m\le N^\beta}
\bbP(\widetilde {E}_m (\ell,h,V)^\complement\cap\widetilde F)\\&\lc \tfrac 12  N^{-h}+ N^{\beta-d-h-1}\le N^{-h}.\end{align*}

It remains to show that
 \Be\label{eq:Vtkaest} V\widetilde \kappa(\ell,\beta,h)\le \frac 1{\beta\ell}
\big(10^{d+1}\ell^2(\ell+h)\big)^\ell\Ee
for $V= e^{2d+10}$. For  $\widetilde \kappa(\ell, \beta, h)$ we
have, by Lemma \ref{le:growth},
\[\widetilde \ka(\ell, \beta, h)
\le
 1+ \sum_{q=1}^{\ell-1}
\frac{\ell}{\ell-q} \binom {\ell-1}{q} \frac{(e^{d+3}q^2 (h+1+q))^q}{\beta(\ell-q)}
\]
and the right hand side is estimated by $(\beta\ell)^{-1}\ka_*(\ell,h)$ where
$\ka_*(\ell,h)$ is the expression in line \eqref{eq:kastarline}.
The estimation that follows in the proof of Lemma \ref{le:growth} yields
\Be\label{eq:sec-kaest}
\widetilde \ka(\ell, \beta, h)\le e^{1/2} \frac{\ell}{\beta}(e^{d+3}\ell^2(h+\ell))^{\ell-1}
\Ee
and thus  clearly
\eqref{eq:Vtkaest} follows.
\end{proof}

\begin{lemma}\label{le:logNvariant}
Let $\ell \in \bbN$,  $h\in \bbN$ and $B\ge 1$. There exist positive constants
$N_0(B,\ell)$ and $M_0(B, \ell, h, d)$  so that for $N\ge N_0(\ell, B)$ the event
$$\max_{m\le (BN^d\log N)^{1/\ell}}\max_{u\in \Gamma_N^d}\sigma_m^{*\ell}(\{u\})\le
M_0 (\ell, B,h,d) \log N$$
has probability at least $1-N^{-h}$.
\end{lemma}

\begin{proof}
If $\ell\ge 2$ we may assume that
\Be\label{N0cond}  B\log N \le N^{\frac{d}{2(\ell-1)}} \text{
for $N\ge N_0(\ell,B)$.}
\Ee

 Let
\Be
\label{eq:hatkappadef}\widehat \ka(\ell, h):=
\sum_{q=0}^{\ell-1}\binom {\ell}{q} M(q, \frac d2(1-\frac q \ell), h+1)\,.
\Ee
and let
\Be \label{cVassu}
\cV\ge 2d+h+1+20B.
\Ee

Let $\widehat E_m(\ell, h,\cV)$
denote  the event that
$$\sigma_m^{*\ell}(\{u\})\le \widehat \kappa(\ell, h) \cV \log N$$
holds true for all $u\in \Gamma_N^d$.
We condition on the event
\Be\label{eq:thatFdef}\widehat F= \bigcap_{q=1}^{\ell-1}
 E(q,\frac d2 (1-\frac q\ell),h+1)\,,
\Ee  again  with the sets on the right hand side defined as in the statement of Lemma \ref{le:ball}.
Then the event $\widehat F^\complement $ has probability at most $\ell N^{-h-1}\le \frac 12 N^{-h}$.

It remains to estimate
$\sum_{m\le (BN^d\log N)^{1/\ell}}
\bbP(\widehat E_m(\ell, h,V)^\complement\cap \widehat F)$.
If we apply the condition $E(q, \frac{d}{2}(1-\frac q\ell),h+1)$ only for cubes of measure $N^{-d}$ then we see that
\Be \sigma_m^{*q}(\{u\})\le M(q, \frac{d}{2}(1-\frac q\ell),h+1), \quad  m\le
N^{\frac d{2q}+\frac d{2\ell}},
\,\,1\le q\le \ell-1.
\Ee
In order to apply it for all $m\le (BN^d\log N)^{1/\ell}$ we must have
$(BN^d\log N)^{1/\ell} \le N^{\frac d{2q}+\frac d{2\ell}}$ which is implied by
\eqref{N0cond}.

By \eqref{binomial-formula-1} we have
$\Delta_{j,\ell}(\{u\})\le \widehat\kappa(\ell, h)$ on $\widehat F$, for $j=1,\dots,m$.
Let $\widehat \cA_{\cV,m}^u$ be the event that
\Be\label{eq:Avmu-event}
\sigma_m^{*\ell}(\{u\})\equiv \sum_{j=1}^m\Delta_{j,\ell}(\{u\})\ge \cV_N \widehat \kappa(\ell, h), \text{ where } \cV_N=\lfloor \cV\log N\rfloor\,
\Ee
and let $\widehat \cA_{\cV,m}$ be the event that
\eqref{eq:Avmu-event} holds for all $u\in \Gamma_N^d$.

 Now we estimate $\widehat \cA_{\cV,m}^u$ on $\widehat F$.
 Notice that if \eqref{eq:Avmu-event} holds on $\widehat F$ there are at least $\cV_N$ indices $j$ so that $\Delta_{j,\ell}(\{u\})\neq 0$ (and we may assume $m\ge \cV_N$).
We argue as in the proof of Lemma \ref{le:ball} using Lemma \ref{lemma:elem}, (ii),
to see that
$$\bbP(\widehat \cA_{\cV,m}^u\cap \widehat F) \le \sum_{k=\cV}^m
\binom{m}{k}(2N^{-d} m^{\ell-1})^k\,.
$$
In order to apply Lemma \ref{small-summation} we must have $\cV_N\ge 2mp$ with $p=2N^{-d} m^{\ell-1}$, and this is certainly satisfied if $\cV\ge 8B$.
Under this condition we thus get
\Be\bbP(\widehat \cA_{\cV,m}^u\cap \widehat F) \le \frac{2(2N^{-d} m^\ell)^{\cV_N}}{V_N!}
\le \frac{2 (2B\log N)^{\cV_N}}{\cV_N!}.\Ee

We use the inequality
\Be \label{factorials} \frac {T^n}{n!} \le e^{-n}, \text{ for $T\ge 1$ and $n\ge e^{2} T$}.
\Ee
To verify  this one takes logarithms and uses $\log (n!)\ge n\log n -n+1$. Thus
the inequality follows from
$n(\log T - \log n) \le -2n$
which is true for $n\ge e^{2}T$.

We  apply \eqref{factorials} with  $T=2B \log N$ and $n=\cV_N$.
Note that
by the assumption \eqref{cVassu} we get
$\cV_N\ge e^2 T$. Therefore
$$
\frac{2 (2B\log N)^{\cV_N}}{\cV_N!}\le 2 e^{-\cV_N} \le 2e^{1-\cV \log N}
\le N^{-(2d+h+1+10 B)}.
$$
Thus
\begin{align*}
&\bbP\big(\cup_{m\le (BN^d\log N)^{1/\ell}}
\widehat E_m(\ell,h,\cV)^\complement\big)\\
\le &
\bbP (\widehat F^\complement) +
\sum_{m\le (BN^d\log N)^{1/\ell}} \sum_{u\in \Gamma_N^d}
\bbP(\widehat \cA_{\cV,m}^u\cap \widehat F)
\\ \le& \frac 12 N^{-h} + (BN^d \log N)^{1/\ell}N^d N^{-10 B} N^{-2d-h-1}\, \le N^{-h}
\end{align*}
and
we get the assertion of the lemma.
\end{proof}

\noi{\it Remark.} It is also possible to give a proof of Lemma \ref{le:logNvariant}
based on the second version of Hoeffding's inequality
\eqref{hoeffdingb} in the appendix (\cf. \cite{chenthesis}).

\medskip

The following proposition can be seen as a discrete analog  to statement (iv) in
Theorem \ref{intro:theorem-C}.

\begin{proposition}
\label{prop:bounded-convolution-full-range}
Given integers $\kappa\ge 1$, $\ell\ge \ka+1$ and $h\ge 1$,  there exists
$N_\ka(\ell,h)\ge 1$ and $M_\ka(\ell,h,d)>0$ such that for all $N\ge N_\ka(\ell,h)$ with $\gcd(\ell!,N)=1$  the event

\Be
\label{eq:bdd-convol-full-range}
\max_{m\le (N^d\log N)^{\frac 1{\ell-\ka}}}\,\max_{u\in \Gamma_N^d}
\frac{\big|\sigma_m^{*\ell}(\{u\})-m^\ell N^{-d}\big|}
{(m^\ell N^{-d} )^{1/2}}
\le M_\kappa(\ell,h,d)
(\log N)^{1+\frac {\ka} 2}
\Ee
has probability at least $1-N^{-h}$.
\end{proposition}

\begin{proof}
We prove this by induction on $\ka$.

\medskip

\noi{\it The case $\ka=1$}.
Let $B_0 \ge d+h+1$, sufficiently large.
We first remark that for  $m^\ell\le B_0 N^d \log N$
inequality
\eqref{eq:bdd-convol-full-range} is implied by
Lemma
\ref{le:logNvariant},
provided that
$N$ is sufficiently large. We thus may assume that
 \Be\label {B0assu} m\ge (B_0 N^d \log N)^{1/\ell}.\Ee

 Following \cite{koerner2008}, we will treat the telescopic sums
 $\sigma_m^{\ell}(\{u\})-m^{\ell}N^{-d}=\sum_{j=1}^m\sigma_j^{*\ell}-\sigma_{j-1}^{*\ell}- N^{-d}(j^\ell-(j-1)^\ell)$
  as a sum of martingale differences with respect to the filtration of
  $\sigma$-algebras  $\mathscr F_j$, with $\mathscr F_j$  generated by the random variables $x_1,\cdots,x_j$, see Lemma \ref{lemma:elem}, (iii).

By Lemma \ref{le:logNvariant}, there is a constant $M_0=M_0(\ell, B_0, h,d) $ so that
\begin{multline} \label{aprioriprob}
\bbP\Big(
\max_{1\le q\le \ell-1}\max_{1\le j\le (B_0 N^d \log N)^{1/q}} \max_{u\in \Gamma_N^d}\sigma_j^{*q}(\{u\}) \le M_0 \log N
\Big) \\ \ge 1-N^{-2(d+h+1)},
\end{multline}
 provided that $N$ is large enough.
Note that \Be (B_0 N^d \log N)^{1/\ell} \le  \min_{1\le q\le \ell-1}
(B_0 N^d \log N)^{1/q}\Ee
provided that $N$ is large enough.
Let $\cE_{j-1}$ denote the event
\Be \label{Ej-1}
\cE_{j-1}= \Big\{
\sigma_{j-1}^{*q}(\{u\}) \le M_0\log N \text{  for $1\le q\le \ell-1$ and all $u\in \Gamma_N^d$}\Big\}.
\Ee
Then \Be \label{Ej-1compl}\bbP \Big(\bigcup_{1\le j  \le
(N^d\log N)^{\frac{1}{\ell-1}}}\cE_{j}^\complement\Big)\le N^{-2(d+h+1)},
\Ee
Define for fixed $u\in \Gamma_N^d$
\[
Y_j\equiv Y_{j,u}:=\begin{cases}
\sigma_j^{*\ell}-\sigma_{j-1}^{*\ell}- N^{-d}(j^\ell-(j-1)^\ell)&\text{ on }
\cE_{j-1},
\\0 &\text{ on }  \cE_{j-1}^\complement.
\end{cases}
\]

We shall apply Lemma \ref{lemma:elem} (iii) to the martingale $\{W_j\}_{j=0}^m$
with $W_0=0$ and  $W_j=\sum_{\nu=1}^j Y_\nu$ for $j\ge 1$.
We prepare for an application of Hoeffding's inequality (Lemma \ref{lemma:hoeffding}) and
estimate the conditional expectation of $e^{\lambda  Y_j}$ given fixed $x_1, \dots,x_{j-1}$.

\Be \label{claimcondexp}
\text{ \it Claim: }\quad
\begin{aligned}
&\text{\it For $|\la|\le (2^\ell M_0 \log N)^{-1} $,}
\\
&\,\bbE[ e^{\lambda Y_j} | \, x_1,\dots, x_{j-1}] \le
 \exp \big(3m^{\ell-1}N^{-d} (2^\ell M_0)^2 (\log N)^2 \la^2\big)\,.
\end{aligned}
\Ee

\noi{\it Proof of \eqref{claimcondexp}.}
Given $(x_1,\dots, x_{j-1}) $, if
inequality \eqref {Ej-1}  does not hold then
 we have  $Y_j=0$ and thus  $\bbE[ e^{\lambda Y_j} | \, x_1,\dots, x_{j-1}] =1$. Thus in this case  \eqref{claimcondexp} holds trivially. We thus need to bound
 \eqref{claimcondexp} on $\cE_{j-1}$.
 First observe
 $$N^{-d} (j^\ell-(j-1)^\ell)\le \ell j^{\ell-1} N^{-d}\le \ell m^{\ell-1} N^{-d}
 \le \ell \log N
 $$ by assumption.
 By \eqref{binomial-formula-1} and
 \eqref{Ej-1},
\[ \sigma_j^{*\ell}(\{u\})-\sigma_{j-1}^{*\ell}(\{u\})\le \sum_{k=0}^{\ell-1} \binom {\ell}{k}M_0\log N \le 2^\ell M_0 \log N\,.\]
Hence we get $|Y_j|\le 2^\ell  M_0\log N$. On the other hand, writing
$$Z_j=\Delta_{j,\ell}(\{u\})= \sigma_j^{*\ell}(\{u\})-\sigma_{j-1}^{*\ell}(\{u\}),$$
we have, by \eqref{binomial-formula-2},
\[\mathbb P(Z_j\neq 0|x_1,\cdots,x_{j-1})
\le N^{-d}\sum_{k=0}^{\ell-1} (j-1)^k \le 2m^{\ell-1}N^{-d}.
\]
We use these observations  to
estimate, for  $0<|\la|\le (2^\ell M_0\log N)^{-1}$, the term
$\bbE[ e^{\lambda Y_j} ]$  which in the following calculation is an abbreviation for the expectation  conditional  on  $x_1,\dots, x_{j-1}$.
Since the expectation of $Y_j$ with respect to  $x_j$ is zero we obtain
\begin{align*}
&\mathbb E [e^{\lambda Y_j}]= \sum_{k=0}^{\infty} \frac{\lambda^k\mathbb E[Y_j^k]}{k!}
= 1+\sum_{k=2}^{\infty} \frac{\lambda^k \mathbb E[Y_j^k]}{k!}=
\\
&1+\mathbb P(Z_j=0)\sum_{k=2}^{\infty} \frac{|\lambda|^k\mathbb E\big[|Y_j|^k\big|Z_j=0\big]}{k!}
+\mathbb P(Z_j\neq 0)\sum_{k=2}^{\infty} \frac{|\lambda|^k\mathbb E\big[|Y_j|^k\big|Z_j\neq 0\big]}{k!}\,.
\end{align*}
We have  $m^{\ell-1} N^{-d}\le \log N$ and thus
\begin{align*}
&\sum_{k=2}^{\infty} \frac{|\lambda|^k\mathbb E\big[|Y_j|^k\big|Z_j=0\big]}{k!}
\le \sum_{k=2} ^\infty \frac{ (|\la|\ell m^{\ell-1} N^{-d})^k}{k!} \\
& \le  (\la \ell m^{\ell-1} N^{-d}) ^2 \sum_{k=2}^\infty \frac{|\la \ell \log N|^k}{k!}
\le  (\la \ell m^{\ell-1} N^{-d})^2\,.
\end{align*}
Also
\begin{align*}
\mathbb P(Z_j\neq 0)\sum_{k=2}^{\infty} \frac{|\lambda|^k\mathbb E\big[|Y_j|^k\big|Z_j\neq 0\big]}{k!}&\le
\mathbb P(Z_j\neq 0)\sum_{k=2}^{\infty} \frac{|\lambda 2^\ell M_0 \log N|^k}{k!}
\\
&\le 2m^{\ell-1}N^{-d}  (\la 2^\ell M_0 \log N)^2.
\end{align*}
Combining the two estimates we get
\begin{align*}\mathbb E [e^{\lambda Y_j}|x_1,\dots,x_{j-1}]
&\le 1+ 3m^{\ell-1} N^{-d}(2^\ell M_0\log N)^2
\\
&\le \exp(3m^{\ell-1} N^{-d}(2^\ell M_0\log N)^2),
\end{align*}
thus proving \eqref{claimcondexp}.

\medskip

We now apply  Hoeffding's inequality
(\cf. \!\eqref{hoeffdinga} in  Lemma \ref{lemma:hoeffding} in the appendix)  with
the parameters
\begin{align*}
a_j^2&= 6m^{\ell-1} N^{-d}(2^\ell M_0\log N)^2
\\
A&= \sum_{j=1}^m a_j^2=6m^{\ell} N^{-d}(2^\ell M_0\log N)^2
\\
\delta &=(2^\ell M_0 \log N)^{-1}
\\
t&= 2
\sqrt{ A(d+h+1)\log N} =  M_1 (m^{\ell} N^{-d})^{1/2}(\log N)^{3/2}
\end{align*} where
$$M_1=
M_0 2^\ell \sqrt{24(d+h+1)}.
$$ For \eqref{hoeffdinga} to hold  we must have $t\le A\delta$ which one checks to be equivalent  with $(d+h+1) \log N\le \frac 32 m^\ell N^{-d}$, and thus valid by
\eqref{B0assu}. Thus, by \eqref{hoeffdinga},
\begin{align*}&\bbP\Big( \Big|\sum_{j=1}^m Y_{j,u}\Big| \ge M_1 (m^{\ell} N^{-d})^{1/2}(\log N)^{3/2} \Big)
\\& \le 2 \exp( -t^2/2A)= 2\exp (- 2(d+h+1)\log N)= 2N^{-2(d+h+1)}\,.
\end{align*}
Allowing $u\in \Gamma_N^d$  and $m\le (N^d\log N)^{\frac{1}{\ell-1}}$ to vary, we see that
\begin{multline}\label{eq:convol-finalstep}
\mathbb P\Big(\max_{u\in\Gamma_N^d}
\max_{m\le (N^d\log N)^{\frac{1}{\ell-1}}}
\frac{\big|\sum_{j=1}^m Y_{j,u}
\big|}{(m^\ell N^{-d})^{1/2}}\ge
M_1 (\log N)^{3/2} \Big)
\\
\le 2 N^{-d-2h-2} N^d (N^d\log N)^{\frac{1}{\ell-1}} \le N^{-2h-1}
\end{multline}
if $N$ is large enough.
Now $\sigma_m^{*\ell}(\{u\})- m^\ell N^{-d} -\sum_{j-1}^mY_{j,u}=0$ on
$\bigcap_{1\le j\le m} \cE_{j-1}$ and thus
\Be\label{eq:convol-finalstep2}
\begin{aligned}
&\mathbb P\Big(\max_{u\in\Gamma_N^d}
\max_{m\le (N^d\log N)^{\frac{1}{\ell-1}}}
\Big|
\sigma_m^{*\ell}(\{u\} -m^\ell N^{-d} -\sum_{j=1}^m Y_{j,u}\Big| \neq 0\Big)
\\&\quad \le
\sum_{1\le j-1\le (N^d \log N)^{\frac{1}{\ell-1}}
} \bbP (\cE_{j-1}^\complement)
\le   (N^d \log N)^{\frac{1}{\ell-1}} N^{-2(d+h+1)} \le N^{-2h-1}
\end{aligned}
\Ee
if $N$ is large enough. This establishes the assertion for $\kappa=1$.

\medskip

\noi{\it The induction step.} We now assume $\ka\ge 2$, $\ell\ge \ka+1$ and that the assertion holds for $1\le \ka'<\ka$. Let $h\ge 1$ and fix $j$ with $1\le j\le (N^d\log N)^{\frac{1}{\ell-\ka}}$.

By Lemma \ref{le:logNvariant} and by the induction hypothesis there exist
$N_{\ka-1}=N_{\ka-1}(\ell)$ and $\cC=\cC_{\ka-1}(\ell, h,d)\ge 1$ so that
for all $N\ge N_{\ka-1}$ the  event
$$E_{j-1}=E_{j-1}(\ell,\ka-1,N),$$
 given  by the following three conditions
 \eqref{Efirst}, \eqref{Esecond}, \eqref{Ethird}
 has probability
at least $1- N^{-2(h+d+1)}$.

\noi{\it Definition of $E_{j-1}$}:
\begin{subequations}\label{Ej-1def}
\begin{align}\label{Efirst}
&\max_{u\in \Gamma_N^d} \sigma_{j-1}^{*q}(\{u\}) \le \cC \log N\,\text{ for $1\le q\le \ell-\ka$.}
\\ \label{Esecond}
& \max_{u\in \Gamma_N^d} \sigma_{j-1}^{*q}(\{u\}) \le \cC \log N\,\quad \text{ for   those $q$ with }
\\\notag&\qquad\qquad
\text{ $\ell-\ka+1\le q\le \ell-1$, and $j-1 \le (N^{d}\log N)^{1/q}$}.
\\
\label{Ethird}
&\max_{u\in \Gamma_N^d}
\Big|\sigma_{j-1}^{*q}(\{u\}) - \frac{(j-1)^q}{N^d}\Big| \le \cC \Big(\frac{(j-1)^q}{N^d}\Big)^{1/2}(\log N)^{1+ \frac{\ka'}{2}}
\\
\notag
&\qquad\text{for those $q$, $\ka'$ with $\ka'<\ka$, $q\le \ell$,  }
\\ \notag &\qquad\text{$(N^d\log N)^{\frac{1}{q-\ka'+1}}\le j-1\le (N^d\log N)^{\frac{1}{q-\ka'}}$.}
\end{align}
\end{subequations}
We define
\[
\Ups_j\equiv \Ups_{j,u}:=\begin{cases}
\sigma_j^{*\ell}(\{u\})-\sigma_{j-1}^{*\ell}(\{u\})- N^{-d}(j^\ell-(j-1)^\ell)&\text{ on }
E_{j-1},
\\0 &\text{ on }  E_{j-1}^\complement
\end{cases}
\]
and claim that
\Be
\label{Upsclaim}
|\Ups_{j,u}|\le  \cC 2^{\ell+2} \Big(\frac{m^{\ell-1}}{N^d}\Big)^{1/2} (\log N)^{\frac{\ka+1}{2}}\,.
\Ee
To see \eqref{Upsclaim} we decompose using \eqref{binomial-formula-1}
\begin{align*}
&\sigma_j^{*\ell}-\sigma_{j-1}^{*\ell} - \frac{j^\ell-(j-1)^\ell}{N^{d}}
\\&\quad= \sum_{q=0} ^{\ell-\ka}\binom {\ell}{q} \delta_{(\ell-q)x_j}* \sigma_{j-1}^{*q}
-
\sum_{q=0} ^{\ell-\ka}\binom {\ell}{q} \frac{(j-1)^q }{N^{d}}
\\ &\quad+
\sum_{q=\ell-\ka+1}^{\ell-1}
\binom{\ell}{q} \, \delta_{(\ell-q)x_j}*\Big( \sigma_{j-1}^{*q}- \frac{(j-1)^q}{N^d}\Big)
\end{align*}
Now we have $m\le (N^d\log N)^{\frac{1}{\ell-\ka}}$ and thus
$\sum_{q=0} ^{\ell-\ka}\binom {\ell}{q} \frac{(j-1)^q }{N^{d}}\le2^\ell m^\ell N^{-d}\le 2^\ell\log N.$
On $E_{j-1}^\complement$ we have by \eqref{Efirst}
$\sum_{q=0} ^{\ell-\ka}\binom {\ell}{q} \delta_{(\ell-q)x_j}* \sigma_{j-1}^{*q}(\{u\})\le 2^\ell \cC \log N.$
If $\ell-\ka+1\le q\le \ell-1$ each $j$ with $j-1\le (N^d\log N)^{\frac{1}{\ell-\ka}}$ satisfies either
$(j-1)\le (N^d\log N)^{1/q}$ or
$(N^d\log N)^{\frac{1}{q-\ka'+1}}< j-1\le (N^d\log N)^{\frac{1}{q-\kappa'}}$ for some $\ka'$ with $1\le \ka'<\ka$.
If $(j-1)\le (N^d\log N)^{1/q}$ we use \eqref{Esecond} to bound
$|\sigma_{j-1}^{*q}\{u\}- \frac{(j-1)^q}{N^d}|$ by $(\cC+1)\log N$.
If $(N^d\log N)^{\frac{1}{q-\ka'+1}}< j-1\le (N^d\log N)^{\frac{1}{q-\kappa'}}$
we use \eqref{Ethird} to bound
$|\sigma_{j-1}^{*q}\{u\}- \frac{(j-1)^q}{N^d}|$ by $\cC((j-1)^qN^{-d} )^{1/2} (\log N)^{1+\kappa'/2}$ and hence by $\cC(m^{\ell-1}N^d)^{1/2} (\log N)^{\frac{\ka+1}{2}}$. Now sum and combine everything  to get
\eqref{Upsclaim}.

Now given \eqref{Upsclaim} we can apply the Azuma-Hoeffding inequality
(Corollary \ref{lemma:azuma}) with
\begin{align*}
a_j&= 2^{\ell+2} \cC
\Big(\frac{m^{\ell-1}}{N^d}\Big)^{1/2} (\log N)^{\frac{\ka+1}{2}},
\\
A&=\sum_{j=1}^m a_j^2= (2^{\ell+2} \cC)^2m^\ell N^{-d} (\log N)^{\ka+1},\\
t&= \sqrt{2A(2d+2h+2)\log N}= M_\ka (m^\ell N^{-d})^{1/2} (\log N)^{1+\frac{\ka}{ 2}}
\end{align*}
with $M_\ka(\ell, h,d)=(2d+2h+2)^{1/2}2^{\ell+2} \cC_{\ka-1}(\ell,h,d)$.
We get
\begin{align*}&\bbP\Big( \Big|\sum_{j=1}^m \Ups_{j,u}\Big| \ge M_\ka (m^{\ell} N^{-d})^{1/2}(\log N)^{1+\frac {\ka}{2}} \Big)
\\& \le 2 \exp( -t^2/2A)= 2\exp (- 2(d+h+1)\log N)= 2N^{-2(d+h+1)}
\end{align*}
To conclude we argue as in the beginning of the induction.
Allowing $u\in \Gamma_N^d$  and $m\le (N^d\log N)^{\frac{1}{\ell-\ka}}$ to vary, we see that
\begin{multline}\label{eq:convol-finalstep3}
\mathbb P\Big(\max_{u\in\Gamma_N^d}
\max_{m\le (N^d\log N)^{\frac{1}{\ell-\ka}}}
\frac{\big|\sum_{j=1}^m Y_{j,u}
\big|}{(m^\ell N^{-d})^{1/2}}\ge
M_\ka (\log N)^{1+\frac {\ka}{2}} \Big)
\\
\le 2 N^{-2d-2h-2} N^d (N^d\log N)^{\frac{1}{\ell-\ka}} \le N^{-2h-1}
\end{multline}
if $N$ is large enough.
Moreover
\begin{align*}
&\bbP\Big( \max_{u\in\Gamma_N^d}
\max_{m\le (N^d\log N)^{\frac{1}{\ell-\ka}}} \Big|\sigma^{*\ell}(\{u\})- \frac{m^\ell}{N^d}
\Big| \ge
M_\ka (\log N)^{1+\frac {\ka}{2}} \Big)
\\ &\le \ell N^{-2h-1} + \sum_{1\le j\le (N^d\log N)^{\frac{1}{\ell-\ka}} }\bbP(E_{j-1}^\complement)\,\le N^{-h}
\end{align*}
if $N\ge N_\ka(\ell)$ large enough.
\end{proof}

\begin{proof}[\it Proof of Proposition \ref{prop:point-masses}]
Let $P=m=\lfloor N^\beta\rfloor $, with $N$ large.
Then the inequalities for
$\sigma_P$ and $P^{-1}\sigma_P$
in
Lemma \ref{fourier-decay-discrete},
Lemma \ref{le:ballvariant}  and
Proposition \ref{prop:bounded-convolution-full-range}
 hold with positive (and high)  probability.
Proposition \ref{prop:point-masses} is an immediate
consequence.
\end{proof}

\section{Fourier restriction and multiplier
 estimates} \label{sec:restrmult}
\subsection{\it Proof of Theorem \ref{intro:theorem-A}}
\label{proofofrestr}
The restriction estimate is equivalent with the bound
\begin{equation}\label{adjrestriction}
\|\widehat{g\mu}\|_{L^{p'}(\bbR^d)}\lesssim\|g\|_{L^2(\mu)}.
\end{equation}
If
 $\mu^{*n}\in L^\infty (\bbR^d)$ then \eqref{adjrestriction} for
 $p=\frac{2n}{2n-1}$  follows
 from  a special case  of an inequality in  \cite{chen}, namely
\Be\label{chenineq}
\|\widehat {g\mu}\|_{2n}^{2n}  \lc
\|\mu^{*n}\|_\infty \|g\|_{L^2(\mu)}^{2n}\,.
\Ee
In conjunction with
Theorem \ref{intro:theorem-C} this  proves
 Theorem \ref{intro:theorem-A}. \qed

\subsection{\it Multipliers of Bochner-Riesz type}\label{sec:M}

For $p\le q\le 2$ we formulate $L^p\to L^q$ versions of the multiplier theorem  \ref{intro:multiplier} stated in the introduction. The main result is

\begin{theorem}\label{prop:multiplier}
Let $1\le p\le q\le 2$, and let
$N>d(1/q-1/2)$ be an integer. Let $\mu$ be  a Borel probability measure on $\mathbb R^d$, and assume that
the Fourier restriction theorem holds,
i.e.
\Be\label{eq:FRL2}
\sup_{\|f\|_p\le 1} \Big(\int|\widehat f|^2d\mu\Big)^{1/2} \le A_p<\infty.\Ee
For  $r\le 1$ let
\Be\label{ups}\varups(r)= \sup_{x\in \bbR^d} \mu(B(x,r)),\Ee
and let  $\eta_r \in C^\infty$ be supported in $\{\xi:
r/4\le|\xi|\le r\}$  and
satisfy the differential inequalities
$r^{|\beta|}\|\partial^\beta \eta_r\|_\infty \le 1$ for all multiindices $\beta$ with $|\beta|\le N.$
Let $$h=\eta_r*\mu.$$
Then, for all $f\in L^p(\bbR^d)$,
\Be\label{multpqestimate}
\big\|\cF^{-1}[ h \widehat f]\big\|_q
\lc r^{d-\frac dq}A_p (\varups(r))^{1/2} \|f\|_p
\Ee
where the implicit constant is independent of $r$ and $\eta$.
\end{theorem}

\begin{proof} The proof is an adaptation of the argument by Fefferman and Stein in
 \cite{fefferman}.
Let $\Phi\in C^\infty(\bbR^d)$  supported in $\{x, |x|\le 1\}$ so that $\Phi(x)=1$ for $|x|\le 1/2$. Let
\begin{align*}
&\Phi_{0,r}(x)= \Phi(rx),\\
&\Phi_{n,r}(x)= \Phi(2^{-n}rx)-\Phi(2^{-n+1}rx), \quad n\ge 1. \end{align*}
 Then we decompose
$h=\sum_{n\ge 0} h_n$ where
$\cF^{-1}[h_n](x)= \cF^{-1}[h](x)\Phi_{n,r}(x).$

We first examine the $L^\infty$ norm of $h_n= h*\widehat \Phi_{n,r}$.
Observe, by the support property of $\eta_r$ and $\|\eta_r\|_\infty\le 1$,
$$|h(\xi)|\le \mu(B(\xi,r)) \le \varups(r).$$
Moreover,
$$|h_{n}(\xi)|\le \varups(r)  \int|\widehat \Phi_{n,r}(y)| dy
 \lc \varups(r)$$
since the $L^1$ norm of
$\widehat \Phi_{n,r}$ is uniformly bounded in $n$ and $r$.
For $n\ge 1$ the last  estimate can be improved since then $ \Phi_{n,r}$ vanishes near $0$ and therefore  all moments of $\widehat \Phi_{n,r}$ vanish. This allows us to write
\begin{align*}
h_n(\xi)&=\int \widehat \Phi_{n,r}(y)\int \big[
\eta_r(\xi-w-y)-\sum_{j=0}^{N-1}\frac{1}{j!}
(\inn{y}{\nabla})^j\eta (\xi-w)\big] d\mu(w)\, dy
\\
&=\int_0^1\frac{(1-s)^{N-1}}{(N-1)!} \int \widehat \Phi_{n,r}(y)\int
 (\inn{y}{\nabla})^N\eta_r (\xi-sy-w) d\mu(w)\, dy\, ds.
\end{align*}
Assuming $N_1>N+d$, this gives
$$|h_n(\xi)|\le C(N_1) \varups(r) \int \big(\frac{|y|}{r}\big)^N
\frac{(2^n/r)^d}{(1+2^n|y|/r)^{N_1}} dy\,
$$
and then
\Be\label{Linftymn}
\|h_n\|_\infty  \le C_N 2^{-nN} \varups (r)\,.
\Ee

Since
$\cF^{-1}[h_n]$ is supported on a ball of radius $2^nr^{-1}$ we get the estimate
\Be \label{q2}\|\cF^{-1}[h_n]*f \|_q
\lc(2^nr^{-1})^{d(\frac 1q-\frac 12)}
\|\cF^{-1}[h_n]*f \|_2\,.
\Ee
To see this one decomposes $f= \sum_Q f_{Q,n}$ where the cubes $Q$ form a grid
of cubes of sidelength $2^n/r$ with $f_Q$ supported in $Q$, and
$\cF^{-1}[h_n]*f $ supported in the corresponding double cube. In view of this support property
$\|\sum_Q\cF^{-1}[h_n]*f \|_q\le C_d
(\sum_Q\|\cF^{-1}[h_n]*f \|_q^q)^{1/q}$ and \eqref{q2} follows by H\"older's inequality.

Next, by Plancherel's theorem,
$$
\|\cF^{-1}[h_n]*f \|_2^2=  \|h_n\widehat f\|_2^2
\le \|h_n\|_\infty
\int |\widehat f(\xi)|^2 |h_n(\xi)| d\xi
$$
and
\begin{align*}
&\int |\widehat f(\xi)|^2 |h_n(\xi)| d\xi\le
\int |\widehat f(\xi)|^2
\int |\eta_r*\widehat\Phi_{n,r}(\xi-w)| d\mu(w)  d\xi
\\
&\,=\int |\eta_r*\widehat\Phi_{n,r}(\xi)| \int
| \widehat f(\xi+w)|^2 d\mu(w)  d\xi \,\le \, A_p^2
\|\eta_r*\widehat \Phi_{n,r}\|_1
\|f\|_p^2
\end{align*}
where for the last inequality we have applied
 the assumed Fourier restriction inequality to
the function $f e^{-i\inn{w}{\cdot}}$.

Now $\|\eta_r*\widehat \Phi_{n,r}\|_1 \lc \|\eta_r\|_1 \lc r^d$ and for
$n\ge 1$, we also get
(using Taylor's theorem as above)
\[
\|\eta_r*\widehat \Phi_{n,r}\|_1 \le
\int
\|\inn{y}{\nabla}^N\eta\|_1
|\widehat \Phi_{n,r}(y)| dy
\lc 2^{-Nn}r^d.\]
The  above estimates  yield
\begin{align*}\|\cF^{-1}[h_n]*f \|_2 &\lc \|m_n\|_\infty^{1/2}
2^{-nN/2}r^{d/2}  A_p \|f\|_p
\\&\lc 2^{-nN} r^{d/2}\sqrt{\varups(r)} A_p\|f\|_p\,,
\end{align*}
by \eqref{Linftymn}. We combine this with  \eqref{q2} to get
$$\|\cF^{-1}[h_n]*f \|_q \lc
2^{-n(N-d(\frac 1q-\frac 12))} r^{d-d/q} \sqrt{\varups(r)} A_p\|f\|_p\,,
$$
and finish by summing in $n$.
\end{proof}

\medskip

As a corollary we get one direction of the statement in Theorem \ref{intro:multiplier}
for the multiplier $m_\la$ as in \eqref{intro:mla}
\begin{corollary}\label{cor:multiplier}
Let  $\mu$ be  a Borel probability measure on $\mathbb R^d$, $\varups$ as in
\eqref{ups} and assume that $\varups(r)\le C_\eps r^{\alpha-\eps}$ for all $\eps>0$.
Let $\chi\in C^\infty_c(\mathbb R^d)$ and define, for  $\la>0$,
$$m_\la(\xi)=\int_{\mathbb R^d}\chi(\xi-\eta)|\xi-\eta|^{\la-\alpha}d\mu(\eta).$$
Assume that  $1\le p\le q\le 2 $ and that \eqref{eq:FRL2} holds.
Then the inequality
\begin{equation}\label{eq:BRmultiplier}
\|\mathcal F^{-1}[m_\la \widehat f]\|_q\lesssim \|f\|_p
\end{equation}
holds for  $\la>d(\frac{1}{q}-\frac{1}{2})-\frac{d-\alpha}{2}$.

If in addition $$\int_0^1 [t^{-\alpha}\varups(t)]^{1/2} \frac{dt}{t}<\infty,$$ then
\eqref{eq:BRmultiplier} holds for
$\la\ge d(\frac{1}{q}-\frac{1}{2})-\frac{d-\alpha}{2}$.
\end{corollary}

\begin{proof}
Decompose $\chi(\xi)|\xi|^{\la-\alpha}= \sum_{j=0}^\infty 2^{-j(\la-\alpha)} \eta_j(\xi)$ where (for a suitable constant $C_N$) the function $C_N^{-1} \eta_j$
satisfies the assumption of Proposition \ref{prop:multiplier} with $r=2^{-j}$.
Thus
$$\|2^{-j(\la-\alpha)}\eta_j*\mu\|_{M_p^q}\lc
2^{-j(\la-d(\frac 1q-\frac 12)+ \frac{d-\alpha}{2})}  \sqrt{2^{-j\alpha} \varups(2^{-j})}.
$$
The corollary follows.
\end{proof}

We now discuss the necessity of the condition on $\la$.
One  may test the convolution operator on a Schwartz function whose Fourier transform equals $1$ on the (compact)  support of $m_\la$.
Therefore,
the condition $m_\la\in M_p^q$ implies
$\cF^{-1}[m_\la]\in L^q$.

\begin{lemma}
Let $\mu$ be a Borel measure supported on a set of Hausdorff dimension $\alpha$ and assume that   $|\widehat\mu(x)|\le C_\gamma(1+|x|)^{-\gamma/2}$ for every  $\gamma<\alpha$.
Let $\la>\alpha-d$, $m_\la$ be as in \eqref{intro:mla},
and $\chi\in C^\infty_c$ with $\widehat \chi$ nonnegative and $\widehat\chi(0)>0$.
Let $K_\la=\cF^{-1} [m_\la]$, $1\le q\le 2$, and assume
$K_\la\in L^q$. Then
$\la\ge d(\frac{1}{q}-\frac{1}{2})-\frac{d-\alpha}{2}$.
\end{lemma}
\begin{proof}
We argue as in Mockenhaupt
\cite{mockenhaupt}.
The positivity conditions on $\chi$ and formulas for fractional integrals imply that for $\ga<\alpha$ there exist $c>0$, $c_\ga>0$, such that   for $|x|\ge 1$
$$|K_\la(x)|\ge c |x|^{\alpha-\la-d}|\widehat \mu(x)| \ge c_\ga
|\widehat \mu(x)|^{1+\frac{2(\la+d-\alpha)}{\ga}}\,.$$
The second inequality follows by the assumption on 
$\widehat\mu$
and $\la>\alpha-d$. The displayed inequality and the condition $K_\la\in L^q$ implies $\widehat\mu\in L^r$, for
$r>q(1+2(\la+d-\alpha)\alpha^{-1})$.
It is shown  in \cite{salem} that $\widehat \mu \in L^r$ implies
$r\ge 2d/\alpha$; indeed this follows from the fact that $\dim_{\rm H}(\supp\,\mu)=\alpha$ implies that the energy integral
$I_\beta(\mu)$ is infinite for $\beta>\alpha$, and H\"older's inequality.
We now have the condition  $\frac{2d}{\alpha}
\le (1+ \frac {2(\la+d-\alpha)}{\ga})q$ which is equivalent with $\la\ge d(\frac{1}{q}-\frac{1}{2})-\frac{d-\alpha}{2}- (\alpha-\ga) (\frac{d}{\alpha q}-\frac 12)\,$. This holds for all $\ga<\alpha$ and the assertion follows.
\end{proof}

\subsection{\it Failure of Ahlfors-David regularity}
\label{failureAD}
Before closing this section, we note that the measures for which the endpoint
$L^{\frac{2d}{2d-\alpha}}\to L^2(\mu)$ restriction estimate holds
cannot be Ahlfors-David regular. This can be seen  as a consequence of a result of Strichartz \cite{strichartz}. For the convenience of the reader we give a short direct proof. We remark that some  related results also appear in the recent thesis by Senthil-Raani \cite{senthil-raani}.

\begin{proposition}\label{prop:strichartz-application}
Let $\mu$ be a Borel probability measure supported on a compact set $E\subset\mathbb R^d$ and let, for $\rho\ge 1$
$$\cB_\rho(\mu)\,=\,
\Big(\int_{\rho\le|\xi|\le 2\rho}
|\widehat\mu(\xi)|^{\frac{2d}{\alpha}} d\xi\Big)
^{\frac{\alpha}{2d}}. $$
Suppose that there exist $0<\alpha<d$ and a  constant $c>0$ such that
$$\mu(B(x,r))\ge c r^\alpha$$
for all $x\in E$ and $0<r<1$.
Then

(i) $\limsup_{\rho\to\infty} \cB_\rho(\mu)>0$.

(ii)  $\cF$ does not extend to a bounded operator from $L^{\frac{2d}{2d-\alpha}}(\bbR^d)$ to $L^2(\mu)$.
\end{proposition}

\begin{proof}
Let $\chi$ be a nonnegative   $C^\infty$ function so that $\chi(x)=1$ for $|x|\le 1$ and $\chi(x)=0$ for $|x|\ge 2$.
Let $R\gg1$ and observe that, by assumption,
\begin{align*}
cR^{-\alpha} &\le \int
\mu(B(x, R^{-1}))  d\mu(x)
\le
\iint  \chi(R(x-y)) d\mu(y) d\mu(x)\\
\,&=\inn{\widetilde \mu*\mu}{\chi(R\cdot)}
=  \int|\widehat\mu(\xi)|^2 R^{-d}\widehat \chi(R^{-1}\xi) d\xi\,.
\end{align*}
and therefore,
$R^{d-\alpha} \le C_N
\int|\widehat\mu(\xi)|^2 (1+R^{-1}|\xi|)^{-N} d\xi
$.
Let $\cA_0=B(0,1)$ and $\cA_j=B(0,2^j)\setminus B(0,2^{j-1})$ for $j\ge 1$.
Then
\begin{align}\notag
R^{d-\alpha} &\le C_N \Big(\int_{\cA_0} |\widehat\mu(\xi)|^2d\xi+
\sum_{j\ge 1} \min\{1, (2^{j-1}R^{-1})^{-N}\} \int_{\cA_j} |\widehat\mu(\xi)|^2d\xi\Big)
\\ \label{upperboundA2j}&\le C_N' \Big(1+
\sum_{j\ge 1} \min\{1, (2^jR^{-1})^{-N}\}
2^{j(d-\alpha)}  \cB_{2^j}(\mu)^2
\Big),
\end{align}
by H\"older's inequality.

Now, in order to prove (i) we argue by contradiction and assume that
(i) does not hold, i.e.
$\lim_{\rho\to\infty}  \cB_\rho(\mu)=0$.
Since $\mu$ is compactly supported the expressions
$\cB_\rho(\mu)$ are all finite and by our assumption it follows that
$\sup_\rho \cB_\rho(\mu) \le B<\infty$. We use
\eqref{upperboundA2j} for some $N>d-\alpha$ and obtain  for $R\ge 1$
$$R^{d-\alpha} \le C_{d,\alpha}\big( 1+B^2  R^{\frac{d-\alpha}2}  +
R^{d-\alpha} \sup_{\rho\ge \sqrt{R}}\cB_{\rho}(\mu)^2
\big)$$ and letting $R\to\infty$ this  yields a contradiction.

To prove (ii) we observe that by duality  \eqref{adjrestriction} holds with
$p'=2d/\alpha$. We take $g\in C^\infty_c$ so that $g=1$ on $\supp(\mu)$, and it follows that $\widehat \mu\in L^{2d/\alpha}$. This in turn implies
$\lim_{\rho\to\infty} \cB_\rho(\mu)=0$ in contradiction to the result in (i).
\end{proof}

\appendix
\section{Some standard probabilistic inequalities}
For the convenience of the reader we include the proof of some standard probabilistic inequalities used in this paper.
We will need the following version of Hoeffding's inequality,
a slight variant of the one in \cite{koerner2008}.

\begin{lemma}\label{lemma:hoeffding}
Let $\{W_j\}_{j=0}^m$ be a bounded real-valued martingale adapted to the filtration $\{\mathscr F_j\}_{j=0}^{m}$.
Suppose that $a_j>0$ for $1\le j\le m$ and that
$$\mathbb E\big[e^{\lambda (W_j-W_{j-1})}
|\mathscr F_{j-1}\big]\le e^{a_{j}^2\lambda^2/2} \text{ for all $|\lambda|<\delta$.}$$
Let  $A=\sum_{j=1}^m a_j^2.$
Then
\begin{subequations}\label{hoeffding}
\begin{align} \label{hoeffdinga}
\bbP\big(|W_m-W_0|\ge t\big)&\le
 2e^{-\frac{t^2}{2A}}, \quad
\text{ $0<t< A\delta$, }
\\  \label{hoeffdingb}
\bbP\big(|W_m-W_0|\ge t\big)&\le
2 e^{A\delta^2/2} e^{-\delta t},\quad\text{ $t\ge A\delta$. }
\end{align}
\end{subequations}
\end{lemma}

\begin{proof}
Observe that, if $0<\lambda<\delta$,
\begin{align*}
\mathbb Ee^{\lambda (W_m-W_0)}
&=\mathbb E\big[e^{\lambda (W_{m-1}-W_0)}\mathbb E[e^{\lambda (W_m-W_{m-1})}
|\mathscr F_{m-1}]\big]
\\&\le e^{a_m^2\lambda^2/2}\mathbb E\big[e^{\lambda (W_{m-1}-W_0)}\big].
\end{align*}
By iterating this step  we get
$\mathbb Ee^{\lambda (W_m-W_0)}\le e^{A\lambda^2/2}$.

Now $\mathbb P\{W_m-W_0\ge t\}
=\mathbb P\{e^{\lambda(W_m-W_0)}\ge e^{\lambda t}\}$ and
 Tshebyshev's inequality gives
$$\mathbb P\{W_m-W_0\ge t\}
\le e^{-\lambda t}\mathbb Ee^{\lambda(W_m-W_0)} \le e^{-\lambda t+A\lambda^2/2}.
$$
If $0<t<A\delta$ we set  $\lambda=t/A$, and if $t>A\delta$ we set $\la=\delta$.
For these choices the displayed inequality gives
\Be\label{Diffget}\mathbb P\big\{W_m-W_0\ge t\big\}\le
\begin{cases}  e^{-\frac{t^2}{2A}}
&\text{ for $0<t< A\delta$, }
\\
 e^{A\delta^2/2} e^{-\delta t}
&\text{ for $t\ge A\delta$. }
\end{cases}
\Ee
Similarly, still for $0<\la< \delta$,
$\mathbb P\big\{W_m-W_0\le -t\big\}
=\mathbb P\{e^{-\lambda(W_m-W_0)}\ge e^{\lambda t}\}$ and  argue as above
to see that
$\mathbb P\big\{W_m-W_0\le -t\big\}$ is also bounded by the right hand side
of \eqref{Diffget}. This implies the asserted inequality.
\end{proof}

To verify the assumption in Lemma
\ref{lemma:hoeffding} the following calculus inequality is useful.
(\cf.  \cite[Lemma~1]{hoeffding}).

\begin{lemma}\label{lemma:expofetx}
Let $X$ be a real-valued random variable with $|X|\le a<\infty$ and $\mathbb E[X|\mathscr F]=0$. Then for any $t\in\mathbb R$,
$$\mathbb E[e^{tX}|\mathscr F]\le e^{a^2 t^2/2}.$$
\end{lemma}
\begin{proof}
Replacing $t$ by $at$  and $X$ by $X/a$ it suffices to consider the
case  $a=1$.

By the convexity of the function $x\mapsto e^{tx}$, for $x\in [-1,1]$ we have
$$e^{tx} \le \frac{1-x}{2}e^{-t}+\frac{x+1}{2} e^t= \cosh t + x\sinh(t)$$
 and thus
$\mathbb E[e^{tX}|\mathscr F]\le \cosh t+ \sinh t\,
\mathbb E[X|\mathscr F]$. The last summand drops by assumption.
 Finally  use that $\cosh t\le e^{t^2/2}$ for all $t\in \bbR$ which
follows  by considering the power series and the inequality $(2k)!\ge 2^k k!$.
\end{proof}

A combination of Lemma \ref{lemma:hoeffding} and Lemma \ref{lemma:expofetx} yields

\begin{corollary}[Azuma-Hoeffding Inequality] \label{lemma:azuma}
Let $\{W_j\}_{j=0}^m$ be a bounded real-valued martingale adapted to filtration $\{\mathscr F_j\}_{j=0}^{m}$. For $1\le j\le m$ let $a_j>0$ and suppose that
  $|W_{j}-W_{j-1}|\le {a_j}$. Writing  $A=\sum_{j=1}^m a_j^2,$
we have
$$\mathbb P\big(|W_m-W_0|\ge t\big)\le 2e^{-\frac{t^2}{2A}}$$
for all $t>0$.
\end{corollary}

As a consequence, we obtain a version of Bernstein's inequality.

\begin{corollary}[Bernstein's inequality] \label{lemma:bernstein}
Let $X_1,\cdots,X_m$ be
complex valued independent
random variables with $\mathbb EX_j=0$ and $|X_j|\le M\in (0,\infty)$ for all $j=1,\cdots,m$. Then,
for all $t>0$
$$\mathbb P\Big(\Big|\frac{1}{m}\sum_{j=1}^m X_j\Big|\ge Mt \Big)\le 4 e^{-{mt^2}/{4}}.$$
\end{corollary}
\smallskip

\noi{\it Acknowledgement.} X. Chen would like to thank Pablo Shmerkin for helpful discussions regarding convolution squares and Fourier decay properties of Cantor sets, and Malabika Pramanik for clarifying an example of R. Kaufman.


\end{document}